\documentclass[11pt]{amsart}

\usepackage{algcompatible}
\usepackage{algorithm}
\usepackage[authoryear]{natbib}

\algnewcommand\RETURN{\STATE \textbf{Return: }}
\usepackage[super]{nth}
\usepackage{longtable}
\RequirePackage[OT1]{fontenc}
\RequirePackage{amsthm,amsmath,amssymb}
\usepackage[colorlinks=true, urlcolor=blue, linkcolor=black, citecolor=black]{hyperref}
\usepackage{graphicx}
\usepackage{csquotes}
\usepackage[utf8]{inputenc}
\usepackage{amsmath}
\usepackage{graphicx}
\usepackage{amsfonts}
\usepackage{amssymb}
\usepackage{amsthm}
\usepackage{enumerate}
\usepackage{mathrsfs} 
\usepackage{hyperref} 
\usepackage[english]{babel}
\usepackage{titlesec}
\usepackage{lscape}
\usepackage{bbm}
\usepackage{tikz}
\makeatletter
\newcommand*\rel@kern[1]{\kern#1\dimexpr\macc@kerna}
\newcommand*\widebar[1]{%
  \begingroup
  \def\mathaccent##1##2{%
    \rel@kern{0.8}%
    \overline{\rel@kern{-0.8}\macc@nucleus\rel@kern{0.2}}%
    \rel@kern{-0.2}%
  }%
  \macc@depth\@ne
  \let\math@bgroup\@empty \let\math@egroup\macc@set@skewchar
  \mathsurround\z@ \frozen@everymath{\mathgroup\macc@group\relax}%
  \macc@set@skewchar\relax
  \let\mathaccentV\macc@nested@a
  \macc@nested@a\relax111{#1}%
  \endgroup
}
\makeatother

\makeatletter
\titleformat{\section}[block]{\large \scshape \filcenter}{\thesection.}{0.5em}{}
\titleformat{\subsection}[runin]{\normalfont \bfseries}{\thesubsection.}{0.5 em}{}
\titleformat{\subsubsection}[runin]{\normalfont \bfseries}{\thesubsubsection.}{0.5 em}{}
\usepackage[authoryear]{natbib}

\makeatother

\newtheorem{thm}{Theorem}
\newtheorem{prop}[thm]{Proposition}
\newtheorem{lemme}{Lemma}

\newtheorem{cor}{Corollary}
\newtheorem{hyp}{Assumption}

\newtheorem{Claim}{Claim}

\renewcommand{\d}{\, \mathrm{d}}
\newcommand{\R}{\mathbb{R}}

\renewcommand{\P}{\mathbb{P}\,}
\newcommand{\N}{\mathbb{N}}

\newcommand{\E}{\mathbb{E}}

\newcommand{\V}{\mathbb{V}}

\newcommand{\F}{\mathscr{F}}

\renewcommand{\L}{\mathbb{L}}
\newcommand{\M}{\mathcal{M}}

\newcommand{\var}{\mathrm{ var }}

\newcommand{\LL}{\mathbb{L}}

\newcommand{\XX}{\mathbb{X}}
\newcommand{\TT}{\mathbb{T}}

\newcommand{\YY}{\mathbb{Y}}

\newcommand{\1}{\mathbbm{1}}

\def\restriction#1#2{\mathchoice
              {\setbox1\hbox{${\displaystyle #1}_{\scriptstyle #2}$}
              \restrictionaux{#1}{#2}}
              {\setbox1\hbox{${\textstyle #1}_{\scriptstyle #2}$}
              \restrictionaux{#1}{#2}}
              {\setbox1\hbox{${\scriptstyle #1}_{\scriptscriptstyle #2}$}
              \restrictionaux{#1}{#2}}
              {\setbox1\hbox{${\scriptscriptstyle #1}_{\scriptscriptstyle #2}$}
              \restrictionaux{#1}{#2}}}
\def\restrictionaux#1#2{{#1\,\smash{\vrule height .8\ht1 depth .85\dp1}}_{\,#2}} 
\usepackage{verbatim}

\makeatletter
\def\hlinewd#1{%
\noalign{\ifnum0=`}\fi\hrule \@height #1 %
\futurelet\reserved@a\@xhline}
\makeatother
\title[Estimating the conditional density by histograms and model selection]{Estimating the conditional density by  histogram type estimators and model selection}
\author{Mathieu Sart}
\date{October, 2016} 
 \address{Univ Lyon, UJM-Saint-Etienne, CNRS, Institut Camille Jordan UMR 5208, F-42023, SAINT-ETIENNE, France} 
\email{mathieu.sart@univ-st-etienne.fr}

\begin{document}
\maketitle
\begin{abstract}
We propose a new estimation procedure of the conditional density for independent and identically distributed data. Our procedure aims at using the data to select   a function among  arbitrary (at most countable) collections of candidates.  By using a deterministic Hellinger distance as loss, we prove that the selected function satisfies a non-asymptotic  oracle type inequality under minimal assumptions on the statistical setting.  We derive an adaptive piecewise constant estimator on a random partition that  achieves the expected rate of convergence over (possibly inhomogeneous and anisotropic) Besov spaces of small regularity. Moreover, we  show that this oracle inequality may lead to a general model selection theorem under very mild assumptions on the statistical setting.  This theorem guarantees the existence of  estimators possessing nice statistical properties under various assumptions on the conditional density (such as smoothness or structural ones).  
\end{abstract}
\section{Introduction}

Let $(X_i,Y_i)_{1\leq i \leq n}$ be $n$ independent and identically distributed random variables defined on an abstract probability space $(\Omega, \mathcal{E}, \P)$ with values in $\XX \times \YY$. We suppose  that  the conditional law  $\mathcal{L} (Y_i \mid X_i)$ admits a density $s (X_i, \cdot)$ with respect to a known $\sigma$-finite measure~$\mu$. In this paper, we address the problem of estimating the conditional density~$s$ on a given subset $A \subset \XX \times \YY$.

When $(X_i,Y_i)$ admits a joint density $f_{(X,Y)}$ with respect to a product measure $\nu \otimes \mu$,  one can rewrite $s$ as
$$s(x,y) = \frac{f_{(X,Y)} (x,y)}{f_X (x)} \quad \text{for all $x,y \in \XX \times \YY$ such that $f_{X} (x) > 0$,}$$
where $f_{X}$ stands for the  density of $X_i$ with respect to $\nu$.
A first approach to estimate $s$ was  introduced in the late 60's by~\cite{Rosenblatt1969}. The idea was to replace the numerator and the denominator of this ratio by Kernel estimators.  We refer to~\cite{hyndman1996} for a study of its asymptotic properties. 
 An alternative point of view is to consider the conditional estimation density problem as a non-parametric regression problem. This   has motivated the definition of local parametric estimators which have been asymptotically studied by~\cite{fan1996, hyndman2002, deGooijer2003}.
 Another approach was proposed by    \cite{faugeras2009}. 
  He showed asymptotic results for his copula based estimator under  smoothness assumptions on the  marginal density of~$Y$ and the copula function. 
  
  The aforementioned procedures depend on some parameters that should be tuned according to the (usually unknown) regularity of $s$. The practical choice  of these parameters is for instance discussed in~\cite{FanetYim2004} (see also the references therein). Nonetheless,  adaptive  estimation procedures are rather scarce in the literature.  
   We can cite the procedure of~\cite{efromovich2007}  which yields an oracle inequality for an integrated $\L^2$ loss. His estimator is  sharp minimax under  Sobolev type constraints.    \cite{bott2015adaptive}   adapted the    combinatorial method of~\cite{devroye1996universally} to the problem of   bandwidth selection in  kernel  conditional density estimation. They   showed that this method allows to select the bandwidth according to the regularity of~$s$  by proving an oracle inequality for  an integrated $\L^1$ loss.  The papers of~\cite{brunel2007, Akakpo2011} are based on the minimisation of a penalized $\L^2$ contrast inspired from the least squares. They established a model selection result for an empirical $\L^2$ loss and then  for an integrated $\L^2$ loss.  These  procedures build  adaptive estimators that may achieve the minimax rates   over  Besov classes.  The paper of  \cite{Chagny2013}  is based on projection estimators,     Goldenshluger and Lepski  methodology and a transformation of the data. She showed an oracle inequality for an integrated $\L^2$ loss from which she deduced that her  estimator is adaptive and reaches the expected rate of convergence under Sobolev  constraints on an auxiliary function. \cite{cohen2011}  gave model selection results for the penalized maximum likelihood estimator for a loss based on  a    Jensen-Kullback-Leibler divergence and under bracketing entropy type assumptions on the models. 

Another estimation procedure that can be found in the literature is the one of $T$-estimation ($T$ for test) developed by~\cite{BirgeTEstimateurs}. It  leads to much more general model selection theorems, which   allows the statistician to model finely the knowledge  he has on the target function   to obtain accurate estimates. 
It is shown in~\cite{Birge2012} that one can  build a $T$-estimator of the  conditional density. We now define the loss used in that paper to compare it with  ours. We suppose henceforth that the distribution of~$X_i$ is absolutely continuous with respect to a known $\sigma$-finite measure $\nu$. Let $f_X$ be its Radon-Nikodym derivative. We denote by   $\L^1_+ (A, \nu \otimes \mu)$   the cone of non-negative integrable functions on~$\XX \times \YY$ with respect to the product measure $\nu \otimes \mu$  vanishing outside~$A$.  \cite{Birge2012} measured the quality of his estimator by means of the Hellinger deterministic distance $\delta$     defined by
$$\delta^2(f,g) = \frac{1}{2} \int_{A} \left( \sqrt{f(x,y)} - \sqrt{g(x,y)} \right)^2 \d \nu  (x) \d \mu (y) \quad \text{for all $f,g \in \L^1_+ (A, \nu \otimes \mu)$.}$$
It is assumed in that paper that the marginal density $f_X$ of $X$ is bounded from below by a positive constant. 
This classical assumption seems natural in the sense that the estimation of~$s$ is better  in regions of high value of $f_X$ than regions of low value as stressed for instance in~\cite{bertin2013}.  In the present paper, we  bypass this assumption by  measuring the quality of our estimators  through the  Hellinger distance $h$ defined by  
$$h^2(f,g) = \frac{1}{2} \int_{A} \left( \sqrt{f(x,y)} - \sqrt{g(x,y)} \right)^2 f_{X} (x) \d \nu (x) \d \mu (y)  \quad \text{for all $f,g \in \L^1_+ (A, \nu \otimes \mu)$.}$$
The marginal density $f_X$ can even vanish,  in contrast to most of the  papers cited above. We  propose a new and  data-driven (penalized) criterion adapted to this  unknown loss.  Its definition is  in the line of the ideas developed in~\cite{BaraudMesure, SartMarkov, RhoEstimation}.   

The main result is an oracle type inequality  for   (at most) countable families of functions of  $ \L^1_+ (A, \nu \otimes \mu)$. This inequality holds true without additional assumptions on the statistical setting. We use it a first time as an alternative to resampling methods to select among  families of piecewise constant estimators.  We  deduce   an adaptive estimator that  achieves the expected rates of convergence over a range of (possibly inhomogeneous and anisotropic) Besov classes, including the ones of  small regularities. A second application of this inequality  leads to a new general model selection theorem under  very mild assumptions on the statistical setting. We  propose $3$ illustrations of this result. The first shows the existence of an adaptive estimator that  attains the expected rate of convergence (up to a logarithmic term) over a very wide range of (possibly inhomogeneous and anisotropic) Besov spaces. This estimator is  therefore able to cope in a satisfactory way with very smooth conditional densities as well as with very irregular ones.  The second illustration  deals with the celebrated regression model. It  shows  that the rates of convergence can be faster than the ones we would obtain under pure smoothness assumptions on $s$ when the data actually obey to a regression model (not necessarily Gaussian). The last illustration  concerns the case where the random variables $X_i$ lie in a high dimensional linear space, say $\XX  = \R^{d_1}$ with $d_1$ large. In this case, we    explain how our procedure can     circumvent  the curse of dimensionality.
 
 The paper is organized as follows. In Section~\ref{SectionSelectionEtHO}, we carry out  the estimation procedure and the oracle inequality. We use it to select among a family of piecewise constant estimators and study the quality of the selected estimator. Section~\ref{SectionModelSelection} is dedicated to the general model selection theorem and its applications.  The proofs are postponed to Section~\ref{SectionPreuves}.

We now introduce the notations that will be used all along the paper.  We set $\N^{\star} = \N \setminus \{0\}$, $\R_+ = [0,+\infty)$, $\R_+^{\star} = (0,+\infty)$. For $x,y \in \R$,  $x \wedge y$ (respectively $x \vee y$) stands for $\min(x,y)$  (respectively $\max (x,y)$). The positive part of a real number $x$ is denoted by $x_+ = x \vee 0$. The distance between a point $x$ and a set $A$ in a metric space $(E,d)$ is denoted by $d(x,A) = \inf_{y \in A} d (x,y)$. The cardinality of a finite set $A$ is denoted by $|A|$. The restriction of a function $f$ to a set $A$ is denoted by $\restriction{f}{A}$.  The indicator function of a set $A$ is denoted by~$\1_A$. The notations $c, C,c',C',c_1, C_1, c_2, C_2, \dots$  are for the constants. These constants may change from line to line.

 \section{Selection among  points and hold-out} \label{SectionSelectionEtHO}
Throughout the paper, $n > 3$ and $A$ is of the form $A = A_1 \times A_2$ with $A_1 \subset \XX$, $A_2 \subset \YY$.

\subsection{Selection rule and main theorem.} \label{SubSectionSelectionRulePoints}   Let  $\mathcal{L} (A,  \mu)$ be the subset of $\L^1_+ (A, \nu \otimes \mu)$ defined by
$$\mathcal{L} (A,  \mu) = \left\{f \in \L^1_+ (A, \nu \otimes \mu), \; \sup_{x \in A_1} \int_{A_2} f(x,y)  \d \mu(y) \leq 1 \right\},$$
and let $S$ be an at most countable subset of $\mathcal{L} (A,  \mu) $. The  aim of this section is to use the data $(X_i,Y_i)_{1 \leq i \leq n}$ in order to select a function $\hat{s} \in S$   close to the  unknown conditional density $s$.    We begin by presenting  the procedure. The underlying motivations will be  further discussed below.

Let $\bar{\Delta}$ be a map on $S$ satisfying 
$$\forall f \in S, \, \bar{\Delta}(f) \geq 1 \quad \text{and} \quad \sum_{f \in S} e^{- \bar{\Delta}(f)} \leq 1.$$
We define the function  $T$ on $S^2$ by
 \begin{eqnarray*}
T(f,f') &=& \frac{1}{n} \sum_{i=1}^n  \frac{\sqrt{f'(X_i,Y_i)} - \sqrt{f(X_i,Y_i)}}{\sqrt{f(X_i,Y_i) + f'(X_i,Y_i)}} \\
& & + \frac{1}{2 n} \sum_{i=1}^n \int_{A_2} \sqrt{f(X_i,y) + f'(X_i,y)} \left( \sqrt{f'(X_i,y)} - \sqrt{f(X_i,y)} \right) d \mu (y) \\
& & + \frac{1}{\sqrt{2} n} \sum_{i=1}^n \int_{A_2} \left(f(X_i,y) - f'(X_i,y) \right) d \mu (y),
\end{eqnarray*} 
where  the convention $0/0 = 0$ is used. We  set for $L > 0$,
$$\gamma (f) = \sup_{f' \in S}  \left\{T  (f,f') - L \frac{\bar{\Delta}(f')}{n} \right\}.$$
We finally define our estimator $\hat{s} \in S$ as any element of $S$ such that
\begin{eqnarray} \label{Criterion1}
\gamma(\hat{s}) + L \frac{\bar{\Delta}(\hat{s})}{n} \leq \inf_{f \in S} \left\{ \gamma(f)  + L \frac{\bar{\Delta}(f)}{n} \right\}+  \frac{1}{n}.
\end{eqnarray} 
  
\paragraph{\textbf{Remarks.}} The definition of $T$ comes from a decomposition of the Hellinger distance  initiated by~\cite{BaraudMesure} and taken back in~\cite{ Baraud2012, SartMarkov,Sart2012, RhoEstimation}.  
We shall  show in the proof of Theorem~\ref{theoremConditionalDensity} that for all  $f, f'  \in \mathcal{L} (A, \mu)$, $\xi > 0$, the two following assertions hold true  with probability larger than $1- e^{- n \xi}$: 
\begin{itemize}
\item If $T(f,f') \geq 0$, then $h^2(s,f') \leq c_1 h^2(s,f) + c_2 \xi$
\item  If $T(f,f') \leq  0$, then $h^2(s,f) \leq c_1 h^2(s,f') + c_2 \xi$.
\end{itemize}
In the above inequalities, $c_1$ and $c_2$ are positive universal constants. The sign of $T(f,f')$ allows thus to know which function among $f$ and $f'$ is the closest to $s$ (up to the multiplicative constant $c_1$ and the remainder term $c_2 \xi$). Note that comparing directly $h^2(s,f)$ to $h^2(s,f')$ is not straightforward in practice since $s$ and $h$ are both unknown to the statistician.  

 The definition of the criterion $\gamma$ looks like   the one proposed in Section~4.1 of~\cite{SartMarkov} for estimating the transition density of a Markov chain as well as the one proposed in~\cite{RhoEstimation}  for estimating one or several densities.     The underlying idea is that $\gamma(f) + L \bar{\Delta}(f)/n$ is roughly between $h^2(s,f)$ and $h^2(s,f) + L \bar{\Delta}(f)/n$. It is thus natural to minimize $\gamma(\cdot) + L \bar{\Delta}(\cdot)/n$ to define an estimator~{$\hat{s}$} of $s$. To be more precise, when $L$ is large enough,  the proof of Theorem~\ref{theoremConditionalDensity} shows that for all $\xi > 0$, the following chained inequalities hold true with probability larger than $1 - e^{- n \xi}$ uniformly for $f \in S$,
\begin{eqnarray*}
(1-\varepsilon) h^2(s,f)  - R_1 (\xi)  \leq \gamma (f) + L \frac{\bar{\Delta}(f)}{n} \leq (1+\varepsilon) h^2(s,f) + 2 L \frac{\bar{\Delta}(f)}{n} + R_2 (\xi)
\end{eqnarray*}
where
\begin{eqnarray*}
R_1 (\xi)  &=&  \inf_{f' \in S} \left\{ (1+\varepsilon) h^2(s,f') + L \frac{\bar{\Delta}(f')}{n} \right\} + c_3 \xi \\
R_2 (\xi) &=&  - (1- \varepsilon) h^2(s,S) + c_4 \xi 
\end{eqnarray*}
for  universal constants $c_3 > 0$, $c_4 > 0$, $\varepsilon \in (0,1)$.  We recall that $h^2(s,S)$ is the square of the Hellinger distance between the conditional density~$s$ and the set~$S$, $h^2(s,S) = \inf_{f \in S} h^2(s,f)$. Therefore, as $\hat{s}$ satisfies (\ref{Criterion1}),
 \begin{eqnarray*}
 (1-\varepsilon) h^2(s,\hat{s}) &\leq& \gamma (\hat{s}) + L \frac{\bar{\Delta}(\hat{s})}{n}  + R_1 (\xi) \\
 &\leq& \inf_{f \in S} \left\{ \gamma (f)  + L \frac{\bar{\Delta}(f)}{n} \right\} + 1/n  + R_1 (\xi) \\
 &\leq& \inf_{f \in S}  \left\{  (1+\varepsilon) h^2(s,f) + 2 L \frac{\bar{\Delta}(f)}{n} \right\} + 1/n + R_1(\xi) + R_2 (\xi).
 \end{eqnarray*}
Rewriting this last inequality and using that $\bar{\Delta} \geq 1$ yields:

\begin{thm} \label{theoremConditionalDensity} 
 There exists a universal constant $L_0$ such that if $L \geq L_0$, any estimator $\hat{s} \in S$ satisfying~{(\ref{Criterion1})} satisfies
 \begin{eqnarray} \label{eqProbabiltySCountable}
 \forall \xi > 0, \quad \P \left[   h^2(s,\hat{s} )  \leq  C_1 \inf_{f \in S} \left\{ h^2 (s,f) + L \frac{\bar{\Delta}(f)}{n}\right\}  + C_2\xi \right] \geq 1 - e^{-n \xi},
 \end{eqnarray} 
where $C_1, C_2$ are  universal positive constants. In particular,
$$ \E\left[  h^2(s,\hat{s} ) \right] \leq  C_3 \inf_{f \in S} \left\{ h^2 (s,f) + L \frac{\bar{\Delta}(f)}{n}\right\},  $$
where $C_3 > 0$ is universal.
\end{thm}
Note that the marginal density $f_X$ influences the performance of the estimator~$\hat{s}$ through the Hellinger loss $h$ only. Moreover, no information on $f_X$ is needed to build the estimator.

We can interpret the condition $\sum_{f \in S} e^{- \bar{\Delta}(f)} \leq 1$  as a (sub)-probability on~$S$. The more complex~$S$, the larger the weights  $\bar{\Delta}(f)$.  When $S$ is finite, one can choose  $\bar{\Delta}(f) = |\log S|$,  and the  above inequality becomes
$$ \P \left[ h^2(s,\hat{s} )  \leq  C_1 \left(h^2(s,S) + L \frac{|\log S|}{n} \right) + C_2 \xi \right] \geq 1 - e^{-n \xi}.$$
The  Hellinger quadratic risk of the estimator $\hat{s}$ can therefore be  bounded from above by a sum of two terms (up to a multiplicative constant): the first one stands for the bias term while the second one stands for the estimation term.

Let us mention that assuming that $S$ is a subset of $\mathcal{L} (A,  \mu) $ is not  restrictive. Indeed, if $f$ belongs to $\mathbb{L}_+^1 (A, \nu \otimes \mu) \setminus \mathcal{L} (A,  \mu) $, we can set
\begin{eqnarray*} 
\pi (f) (x,y) &=& \begin{cases} \frac{f (x,y)}{\int_{A_2} f (x,t) \d \mu(t)} &\mbox{if }  \int_{A_2} f (x,t) \d \mu(t) > 1 \; \text{and} \;  \int_{A_2} f (x,t) \d \mu(t)  < \infty \\ 
f (x,y)  & \mbox{if }  \int_{A_2} f (x,t) \d \mu(t) \leq 1 \\
0  &  \mbox{if }  \int_{A_2} f (x,t) \d \mu(t) = \infty.
\end{cases}  \\
\end{eqnarray*}
The  function $\pi(f)$  belongs to $ \mathcal{L} (A,  \mu) $ and does always better than $f$:
\begin{prop} \label{propRisqueDeStilde}
For all $f \in \L^1_+ (A, \nu \otimes \mu)$,   $$ h^2(s, \pi(f)) \leq  h^2(s, f).$$  
\end{prop} 
Thereby, if $S$ is only assumed to  be a subset of  $\mathbb{L}_+^1 (A, \nu \otimes \mu)$,  the procedure applies with $S' = \left\{\pi(f), \, f \in S\right\} \subset \mathcal{L} (A, \mu)$  in place of $S$ (and with  $\bar{\Delta}(\pi(f)) = \bar{\Delta}(f)$). The resulting estimator $\hat{s} \in S'$  then satisfies (\ref{eqProbabiltySCountable}).

Remark:  the procedure does not depend on the dominating measure $\nu$. However, the set~$S$,  which must be chosen by the statistician,  must satisfy the above assumption $S \subset \mathbb{L}_+^1 (A, \nu \otimes \mu)$, which usually requires the knowledge of $\nu$. Actually, this assumption can be  slightly strengthened to deal with an unknown, but finite measure~$\nu$.  This may be of interest when   $\nu$  is   the (unknown) marginal distribution of~$X_i$ (in which case $f_X = 1$).
 More precisely, let    $\L^1_{+,sup} (A,  \mu)$  be  the set of  non-negative measurable functions vanishing outside  $A$ such that
$$\sup_{x \in A_1} \int_{A_2} f(x,y)  \d \mu(y) < \infty.$$
The assumption $S \subset \L^1_{+,sup} (A,  \mu)$  can be satisfied without knowing $\nu$ and implies   $S \subset \mathbb{L}_+^1 (A, \nu \otimes \mu)$.

\subsection{Hold-out.} \label{SectionHO}  As a first application of our oracle inequality, we consider the situation in which   the set  $S$ is a family of estimators built on a preliminary sample. We suppose therefore that we have at hand two independent samples of $Z = (X,Y)$: $\mathbf{Z}_1 = (Z_1,\dots, Z_n)$ and $\mathbf{Z}_2 = (Z_{n+1}, \dots, Z_{2 n})$. This is  equivalent  to  splitting an initial sample  $(Z_1,\dots, Z_{2 n})$ of size $2 n$ into two equal parts:  $\mathbf{Z}_1$ and   $\mathbf{Z}_2$.

Let  $\hat{S} = \{\hat{s}_{\lambda}, \, \lambda \in \Lambda\} \subset \L^1_+ (A, \nu \otimes \mu)$ be an at most countable collection of estimators  based only on the first sample $\mathbf{Z}_1$. 
In view of Proposition~\ref{propRisqueDeStilde}, we may assume, without loss of generality, that for all $\lambda \in \Lambda$, 
\begin{eqnarray*} 
\forall x \in A_1, \; \int_{A_2} \hat{s}_{\lambda} (x,y) \d \mu (y) \leq 1.
\end{eqnarray*} 
Let $\Delta \geq 1$ be a map defined on $\Lambda$  such that $\sum_{\lambda \in \Lambda} e^{-\Delta(\lambda)} \leq 1$.

Conditionally to $\mathbf{Z}_1$, $\hat{S}$ is a deterministic set. We can therefore apply our  selection  rule   to  $S = \hat{S}$, $\bar{\Delta}(\hat{s}_{\lambda}) = \Delta (\lambda)$ and to the sample $\mathbf{Z}_2$ to derive an estimator~$\hat{s}$ such that:
$$\forall \xi > 0, \quad \P \left[  h^2(s,\hat{s} )  \leq  C_1  \inf_{\lambda \in \Lambda} \left\{ h^2 (s, \hat{s}_{\lambda}) + L \frac{\Delta({\lambda})}{n}\right\}  + C_2 \xi  \Bigm|  \mathbf{Z}_1 \right] \geq 1 - e^{-n \xi }.$$
By taking the expectation with respect to $\mathbf{Z}_1$, we then deduce:
\begin{eqnarray} \label{eqDansHoldOut}
\forall \xi > 0, \quad \P \left[  h^2(s,\hat{s} )  \leq  C_1 \inf_{\lambda \in \Lambda} \left\{ h^2 (s, \hat{s}_{\lambda}) + L \frac{\Delta({\lambda})}{n}\right\}  + C_2 \xi  \right] \geq 1 - e^{-n \xi }.
\end{eqnarray} 
Note that there is almost no assumption on the preliminary estimators. It is only assumed that $\hat{s}_{\lambda} \in  \L^1_+ (A, \nu \otimes \mu)$. Besides, the non-negativity of $\hat{s}_{\lambda}$ can always be fixed by taking its  positive part  if needed. We may therefore select among Kernel estimators (to choose the bandwith for instance), local polynomial estimators, projection estimators\dots It is also possible to  mix in the collection  $\{\hat{s}_{\lambda}, \, \lambda \in \Lambda\}$ several type of estimators. From a numerical point of view, the procedure can be implemented in practice provided that $|\Lambda|$ is finite and not too large.

We shall illustrate this  result by applying it to some families of piecewise constant estimators.  As we shall see,  the resulting estimator $\hat{s}$  will be  optimal and adaptive  over some range of possibly anisotropic Hölder and possibly inhomogeneous Besov classes.

 \subsection{Histogram type estimators.} We now define the piecewise constant estimators.
Let $m$ be a (finite) partition of  $A \subset \XX \times \YY$,  and
$$\hat{s}_m (x,y)= \sum_{K \in m} \frac{\sum_{i=1}^n \1_{K} (X_i,Y_i)}{\sum_{i=1}^n \left(\delta_{X_i} \otimes \mu\right) (K)} \1_{K} (x,y),$$
where the conventions $0/0 = 0$, $x / \infty = 0$ are used. \cite{gyorfi2007} established an integrated $\L^1$ risk bound for $\hat{s}_m$ under Lipschitz conditions on $s$.   We are nevertheless unable to find in the literature a non-asymptotic risk bound for the Hellinger deterministic loss~$h$.  We propose the following result (which is assumption free on $s$):
 
 \begin{prop} \label{propRisqHelIntegreeHisto1}
 Let $m$ be a (finite) partition  of $A$ such that each $K \in m$ is of the form $I \times J$ with   $I \subset A_1$, $J \subset A_2$ and $\mu(J) < \infty$.  
 Let  $V_m$ be the cone of non-negative piecewise constant functions on the partition~$m$ defined by
 $$V_m = \left\{ \sum_{K \in m} a_K \1_K, \;  \forall K \in m, \, a_K \in [0, +\infty) \right\}.$$
 Then,
 \begin{eqnarray*}
  \E \left[h^2  (s , \hat{s}_m) \right]  \leq 4 h^2(s, V_m) + 4 \frac{|m|}{n}.
  \end{eqnarray*}
 \end{prop}
 This result shows that the Hellinger quadratic risk $h^2  (s,  \hat{s}_m) $ of the estimator $ \hat{s}_m$ can be bounded by a sum of two terms. The first one  $h^2(s, V_m)$ corresponds to a bias term whereas the second one $|m|/n$ corresponds to a variance or estimation term.    
A deviation bound  can also be established for some partitions:
 \begin{prop}  \label{propRisqHelIntegreeHisto2}
Assume that $m$ is a (finite) partition of $A$ of the form
$$m =  \left\{I \times J ,\, I \in \mathcal{I}, \, J \in \mathcal{J}_I \right\},$$ 
where $\mathcal{I}$ is a (finite) partition of $A_1$, and, for each $I \in \mathcal{I}$, $\mathcal{J}_I$ is a (finite) partition of $A_2$ such that $\mu (J) < \infty$ for all $J \in \mathcal{J}_I$. 

 Then, there exist  universal constants $C_1, C_2 > 0$ such that  for all $\xi > 0$, 
\begin{eqnarray*} 
\P \left[h^2  (s, \hat{s}_m)   \leq  4 h^2(s,V_m) + C_1 \frac{|m|}{n}  + C_2 \xi  \right] \geq 1- e^{-n \xi}.
\end{eqnarray*}
 \end{prop}

 \subsection{Selecting among piecewise constant estimators by Hold-out.} The risk of a histogram type estimator $\hat{s}_m$ depends on the choice of the partition $m$: the thinner $m$, the smaller the bias term $h^2(s, V_m)$ but the larger the variance term $|m|/n$. Choosing a good partition $m$, that is a partition that realizes a good trade-off between the bias and variance terms  is difficult  in practice since  $h^2(s, V_m)$ is unknown  (as it involves the unknown conditional density $s$ and the unknown distance $h$). Nevertheless, combining (\ref{eqDansHoldOut}) and  Proposition~\ref{propRisqHelIntegreeHisto1} immediately entails the following corollary.
 
\begin{cor} \label{corrolarySelectionHisto} 
Let $\mathcal{M}$ be an at most countable collection of finite partitions $m$ of $A$.
Assume that each $K \in m$ is of the form $I \times J$ with   $I \subset A_1$, $J \subset A_2$ and $\mu(J) < \infty$.
Let $\Delta \geq 1$ be a  map on~$\mathcal{M}$ satisfying $$\sum_{m \in \mathcal{M}} e^{- \Delta(m)} \leq 1.$$
Then, there exists an estimator $\hat{s}$ such that
\begin{eqnarray} \label{eqInegaliteOracleConstant}
\E \left[ h^2 \left(s, \hat{s} \right)  \right] \leq  C \inf_{m \in  \mathcal{M}} \left\{   h^2 \left(s, V_m \right)   +  \frac{|m| + \Delta(m)}{n}   \right\},
\end{eqnarray}
where $C$ is a universal positive constant.
\end{cor} 
The novelty of this oracle inequality lies in the fact that it holds for an (unknown) deterministic Hellinger loss under  very mild assumptions both on the partitions and the statistical setting. We avoid   some classical assumptions that are required in the literature to prove similar inequalities (see, for instance, Theorem~3.1 of~\cite{Akakpo2011} for a result with respect to a $\L^2$ loss).

\subsection{Minimax rates over Hölder and Besov spaces.} \label{SectionMinimaxHolderBesov} We can now deduce from (\ref{eqInegaliteOracleConstant})  estimators with nice statistical properties under smoothness assumptions on the conditional density.    Throughout this section,  $\XX \times \YY = \R^d$, $A = [0,1]^d$ and  $\mu$ is the Lebesgue measure.

\subsubsection{Hölder spaces.} \label{SectionDefinitionHolderSpace} Given $\alpha \in (0,1]$, we recall that the  Hölder space $\mathcal{H}^{\alpha} ([0,1])$  is the set of  functions $f$  on $[0,1]$ for which there exists $|f|_{\alpha} > 0$ such that 
\begin{eqnarray} \label{conditionFHolder}
|f(x) - f(y) | \leq |f|_{\alpha}|x - y|^{\alpha} \quad \text{for all $x,y \in [0,1]$.}
\end{eqnarray} 
 Given $\boldsymbol{\alpha} = (\alpha_1, \dots, \alpha_d) \in (0,1]^d$,   the Hölder space $\mathcal{H}^{\boldsymbol{\alpha}} ([0,1]^d)$  is the set of   functions~$f$  on $[0,1]^d$ such that   for all $(x_1,\dots,x_d) \in (0,1]^d$,  $j \in \{1,\dots, d\}$, 
 $$f_j (\cdot) = f(x_1, \cdots, x_{j-1}, \cdot , x_{j+1}, \dots x_{d}) $$
satisfies (\ref{conditionFHolder}) with some constant $|f_j|_{\alpha_j}$ independent of  $x_1,\dots,x_{j-1},x_{j+1},\dots, x_d$. 
We then set $|f|_{\boldsymbol{\alpha}} = \max_{1 \leq j \leq d} |f_j|_{\alpha_j}$.
When all the $\alpha_j$ are equals, the Hölder space $\mathcal{H}^{\boldsymbol{\alpha}} ([0,1]^d)$ is said to be isotropic and anisotropic otherwise.  

Choosing suitably the collection $\mathcal{M}$ of partitions allows to bound from above the right-hand side of (\ref{eqInegaliteOracleConstant}) when  $\restriction{\sqrt{s}}{[0,1]^d}$ is Hölderian. More precisely,  for each integer $N \in \N^{\star}$, let  $m_N$ be the regular partition of $[0,1]$ with $N$ pieces
$$m_{{N}} = \left\{\left[0, {1}/{N}\right[, \left[{1}/{N}, {2}/{N}\right[, \dots,  \left[{(N-1)}/{N}, 1\right] \right\}.$$
 We may define for each multi-integer  $\textbf{N} = (N_1, \dots, N_d) \in (\N^{\star})^d$,
$$m_{\textbf{N}} = \left\{\prod_{j=1}^d I_j, \quad \forall j \in \{1,\dots, d\}, \;  I_j \in m_{N_j} \right\}.$$
 We now choose $\mathcal{M} = \left\{m_{\textbf{N}}, \, \textbf{N} \in (\N^{\star})^d \right\}$, $\Delta(m_{\textbf{N}}) = |m_{\textbf{N}}|$ to deduce  (see, for instance,     Lemma~4 and Corollary~2 of~\cite{BirgePoisson} among numerous other references):
\begin{cor}  
 There exists an estimator $\hat{s}$  such that  for all $ {\boldsymbol{\alpha}} \in (0,1]^d $ and $\restriction{\sqrt{s}}{[0,1]^d} \in \mathcal{H}^{\boldsymbol{\alpha}} ([0,1]^ {d})$, 
\begin{eqnarray*}
\E \left[h^2 \left(s, \hat{s} \right) \right] \leq C \left[ \left|\restriction{\sqrt{s}}{[0,1]^d}\right|_{ {\boldsymbol{\alpha}}}^{\frac{2 d}{d +  2  \bar{\boldsymbol{\alpha}}}} n^{-\frac{2  \bar{\boldsymbol{\alpha}}}{2  \bar{\boldsymbol{\alpha}} + d}}   + n^{-1} \right],
\end{eqnarray*}
where
$ \bar{\boldsymbol{\alpha}}$ stands for the harmonic mean of $ {\boldsymbol{\alpha}}$
$$\frac{1}{ \bar{\boldsymbol{\alpha}}} = \frac{1}{d} \sum_{i=1}^d \frac{1}{\alpha_i},$$
and where $C$ is a positive constant depending only on $d$.
\end{cor}  
The estimator $\hat{s}$ achieves therefore the optimal rate of convergence over the anisotropic Hölder classes $\mathcal{H}^{\boldsymbol{\alpha}} ([0,1]^ {d})$,  $\boldsymbol{\alpha} \in (0,1]^d $. It is moreover adaptive since  its construction does not involve the smoothness parameter $\boldsymbol{\alpha} $.

\subsubsection{Besov spaces.}  The preceding result may be generalized to the Besov classes under a mild assumption on the design density.

We refer to Section~2.3 of~\cite{Akakpo2012} for a precise definition of the Besov  spaces. According to  the  notations developed in this paper,     $\mathcal{B}_{q}^{\boldsymbol{\alpha}} (\L^p ([0,1]^d))$  stands for the Besov space with parameters  $p > 0$, $q > 0$, and smoothness index $\boldsymbol{\alpha} \in (\R_+^{\star})^d$. We denote  its  semi norm by $|\cdot|_{\boldsymbol{\alpha},p, q}$.  
 This space is said to be homogeneous when $p \geq 2$ and inhomogeneous otherwise. It is said to be isotropic when all the $\alpha_j$ are equals and anisotropic otherwise. We now set for $p \in (0,+\infty]$,
\begin{equation*}
 \mathscr{B}^{\boldsymbol{\alpha}} (\L^p ([0,1]^{d}))  =
  \begin{cases} 
  \mathscr{B}^{\boldsymbol{\alpha}}_{\infty} (\L^p ([0,1]^{d})) & \text{if $p \in (0,1]$} \\
   \mathscr{B}^{\boldsymbol{\alpha}}_{p} (\L^p ([0,1]^{d}))  & \text{if $p \in (1,2)$} \\
  \mathscr{B}^{\boldsymbol{\alpha}}_{\infty} (\L^p ([0,1]^{d})) & \text{if $p \in [2,+\infty)$} \\   
   \mathcal{H}^{\boldsymbol{\alpha}} ([0,1]^{d})  & \text{if $p = \infty$} 
  \end{cases}
\end{equation*}
and denote by $|\cdot|_{\boldsymbol{\alpha},p}$ the semi norm associated to the space $\mathscr{B}^{\boldsymbol{\alpha}} (\L^p ([0,1]^{ d}))$. 
 
 The algorithm of~\cite{Akakpo2012} provides a collection $\M$ of partitions $m$  that allows to bound  the right-hand side of (\ref{eqInegaliteOracleConstant}) from above when  $\restriction{\sqrt{s}}{[0,1]^d}$ belongs to a Besov space. More precisely:
 
 \begin{cor}  \label{CorVitesseBesov1}
 Suppose that the  (possibly unknown) density $f_X$   of $X_i$ is upper bounded by a  (possibly unknown) constant $\kappa$ and that $\nu $ is the Lebesgue measure.
 
 Then, there exists an estimator $\hat{s}$ such that,  for all $p \in (2 d / (d+2), +\infty]$, $\boldsymbol{\alpha} \in (0,1)^d$, $\bar{\boldsymbol{\alpha}} > d (1/p-1/2)_+ $ and $\restriction{\sqrt{s}}{[0,1]^d} \in \mathscr{B}^{\boldsymbol{\alpha}} (\L^p ([0,1]^{d}))$, \begin{eqnarray} \label{EqVitesseHisto1}
\E \left[h^2 \left(s , \hat{s} \right) \right] \leq C \left[ \left|\restriction{\sqrt{s}}{[0,1]^d}\right|_{\boldsymbol{\alpha},p}^{\frac{2 d}{d +  2 \bar{\boldsymbol{\alpha}}}} n^{-\frac{2 \bar{\boldsymbol{\alpha}}}{  2 \bar{\boldsymbol{\alpha}} + d}}   +  n^{-1} \right],
\end{eqnarray}
where $C > 0$ depends only on $\kappa$,$d$,$\boldsymbol{\alpha}$,$p$ and where $ \bar{\boldsymbol{\alpha}}$ denotes the harmonic mean of $\boldsymbol{\alpha}$.
\end{cor}
Remark: the control of the bias term   $h(s,V_m)$ in (\ref{eqInegaliteOracleConstant}) naturally involves a smoothness assumption on the square root of $s$ instead of $s$. However, the regularity of the square root of $s$ may be deduced from that of $s$.  Indeed, we can prove that if $s \in \mathscr{B}^{\boldsymbol{\alpha}}_{q} (\L^p ([0,1]^{d}))$ with $\boldsymbol{\alpha} \in (0,1)^d$  then $\sqrt{s} \in \mathscr{B}^{\boldsymbol{\alpha}/2}_{2 q} (\L^{2 p} ([0,1]^{d}))$ and $|\sqrt{s}|_{\boldsymbol{\alpha}/2,2 p,2 q} \leq \sqrt{|s|_{\boldsymbol{\alpha},p,q}}$.  If, additionally,  $s$ is positive on $[0,1]^d$,  then $\sqrt{s}$ also belongs to $\mathscr{B}^{\boldsymbol{\alpha}}_{q} (\L^p ([0,1]^{d}))$ and $$|\sqrt{s}|_{\boldsymbol{\alpha}, p,q} \leq \frac{|s|_{\boldsymbol{\alpha},p,q}}{2 \sqrt{\inf_{x \in [0,1]^d} s(x)}}.$$
Under the assumption of Corollary~\ref{CorVitesseBesov1}, we deduce  that if  $s \in \mathscr{B}_{\infty}^{\boldsymbol{\alpha}} (\L^p ([0,1]^{d}))$ for some   $p \in (2 d / (d+2), +\infty]$, $\boldsymbol{\alpha} \in (0,1)^d$, $\bar{\boldsymbol{\alpha}} > d (1/p-1/2)_+ $,  
\begin{eqnarray*}
  \E \left[h^2 (s, \hat{s}) \right] \leq  C \min \left\{ \left(\frac{\left|\restriction{{s}}{[0,1]^d}\right|_{\boldsymbol{\alpha},p,\infty}}{\left(\inf_{x \in [0,1]^d} s(x)\right)^{1/2}}\right)^{\frac{2 d}{d +  2 \bar{\boldsymbol{\alpha}}}} n^{-\frac{2 \bar{\boldsymbol{\alpha}}}{  2 \bar{\boldsymbol{\alpha}} + d}}  ,    \left|\restriction{{s}}{[0,1]^d}\right|_{\boldsymbol{\alpha}, p, \infty}^{\frac{d}{d+\bar{\boldsymbol{\alpha}}}} n^{-\frac{\bar{\boldsymbol{\alpha}}}{\bar{\boldsymbol{\alpha}} + d}}, n^{-1}  \right\},
\end{eqnarray*}
where $C > 0$ depends only on $\kappa$,$d$,$\boldsymbol{\alpha}$,$p$.

\section{Model selection } \label{SectionModelSelection}
The construction of adaptive and optimal estimators over Hölder and Besov classes follows from the oracle inequality (\ref{eqInegaliteOracleConstant}). This inequality is itself deduced from Theorem~\ref{theoremConditionalDensity}.  Actually, this latter theorem can be applied in a different way to deduce  a more general oracle inequality.  We can then derive adaptive and (nearly) optimal estimators  over more general classes of functions.

\subsection{A general model selection theorem.} 
From now on, the following assumption holds.
\begin{hyp} \label{hypDensiteDesXiCasGeneral}
The (possibly unknown) density $f_X$ of $X_i$ is bounded above by a (possibly unknown) constant $\kappa$. Moreover, $\nu (A_1) \leq 1$.
\end{hyp}

Let  $\LL^2 (A,  \nu \otimes \mu)$ be the space of square integrable functions on $A$ with respect to the product measure $\nu \otimes \mu$ endowed with the distance
$$d_2^2 (f,f') =  \int_{A}  \left(f (x,y) - f'(x,y) \right)^2 \d \nu (x) \d \mu (y) \quad \text{for all $f,f' \in \LL^2 (A, {\nu} \otimes \mu)$.}$$
We say that a subset $V$ of  $\LL^2 (A,  \nu \otimes \mu)$ is a model if it is a finite dimensional linear  space.

The discretization trick described in Section~4.2 of~\cite{SartMarkov} can be adapted to our statistical setting. It leads to the  theorem below. 
 
 \begin{thm} \label{thmSelectionModelGeneral}
 Suppose that Assumption~\ref{hypDensiteDesXiCasGeneral} holds. Let $\V$ be an at most countable collection of models.
Let  $\Delta \geq 1$ be a map on $\V$ satisfying $$\sum_{V \in \V} e^{- \Delta (V)} \leq 1.$$
Then, there exists an estimator $\hat{s}$ such that for all $\xi > 0$
\begin{eqnarray} \label{eqGeneraleSelectionModele}
\quad \P \left[  h^2 (s , \hat{s})  \leq   C \left( \inf_{V \in \V} \left\{\kappa d_2^2 \left( {\sqrt{s}}, V \right) + \frac{\Delta (V) +  \dim (V) \log   n }{n} \right\} +  \frac{\kappa}{n^2} +   \xi \right) \right] \geq 1 - e^{- n  \xi},
\end{eqnarray} 
where $C > 0$ is universal. In particular,
$$   \E \left[h^2 (s , \hat{s}) \right]  \leq  C' \inf_{V \in \V} \left\{d_2^2 \left( {\sqrt{s}}, V \right) + \frac{\Delta (V) + \dim (V) \log  n }{n} \right\}, $$
where $C' > 0$ depends only on $\kappa$.
 \end{thm}
As in Theorem~\ref{theoremConditionalDensity},  the condition $\sum_{V \in \V} e^{- \Delta(V)} \leq 1$ has a Bayesian flavour since it can be interpreted as a (sub)-probability on~$\V$. When $\V$ does not contain too many models per dimension, we can set $\Delta(V) = (\dim V) \log n$, in which case  (\ref{eqGeneraleSelectionModele}) becomes  
 $$\P \left[ h^2 (s , \hat{s})  \leq  C'' \left( \inf_{V \in \V} \left\{\kappa d_2^2 \left( {\sqrt{s}}, V \right) + \frac{\dim (V) \log   n }{n} \right\} +  \frac{\kappa}{n^2}  + \xi \right) \right] \geq 1 - e^{- n \xi},$$
 where $C''$ is universal.  

 This theorem is more general than Corollary~\ref{corrolarySelectionHisto} since it enables us to deal with more general models~$V$.  Moreover,  it provides a deviation bound for $h^2 (s, \hat{s})$, which is not the case of Corollary~\ref{corrolarySelectionHisto}. As a counterpart, it requires an assumption on the marginal density~$f_X$ and  the  bound  involves  a logarithmic term and $\kappa$. 

 Another difference between this theorem and  Corollary~\ref{corrolarySelectionHisto} lies in the computation time of the estimators. The estimator of  Corollary~\ref{corrolarySelectionHisto} may  be built in practice in a reasonable amount of time if  $|\mathcal{M}|$ is not too large. On the opposite, the procedure leading to the above estimator (which is described in the proof of the theorem) is numerically very expensive, and it is unlikely that it could be implemented in a reasonable amount of time. This estimator should therefore be only considered for theoretical purposes.

  \subsection{From model selection to  estimation.}   It is recognized that a model selection theorem such as Theorem~\ref{thmSelectionModelGeneral} is a bridge between statistics and approximation theory. Indeed, it remains to choose  models  with  good approximation properties with respect to the assumptions we wish to consider on $s$   to automatically derive  a good estimator $\hat{s}$.  
 
     A convenient way to model these assumptions is to  consider a class   $\F$  of functions of $\L^2 (A, \nu \otimes \mu)$ and to suppose that $\restriction{\sqrt{s}}{A}$ belongs to $\F$. The aim is then to choose $(\V, \Delta)$ and  to bound 
$$\varepsilon_{\F} (f) = \inf_{V \in \V} \left\{d_2^2 \left (f, V \right) + \frac{\Delta (V) +  \dim (V) \log   n }{n} \right\} \quad \text{for all $f \in \F$}$$ 
from above since 
\begin{eqnarray*}
 \E \left[ h^2(s, \hat{s})  \right] &\leq& C' \varepsilon_{\F} (\sqrt{s}) \\
 \P \left[ h^2(s, \hat{s}) \leq  C''  \varepsilon_{\F} (\sqrt{s}) +  C''' \xi \right] &\geq& 1 - e^{- n \xi} \quad \text{for all $\xi > 0$} 
\end{eqnarray*}   
where $C', C''$ depend only on $\kappa$ and  where $C'''$ is universal.
This work has already been carried out in the literature  for different classes~$\F$ of interest.  
The flexibility of our approach enables the study of various assumptions as illustrated by the three examples below.
 We  refer to~\cite{SartMarkov, BaraudComposite} for  additional examples. In the remainder  of this section, $\mu$ and  $\nu$ stand for the Lebesgue measure.
  
\subsubsection*{Besov classes.} We suppose  that  $\XX \times \YY = \R^d$, $A = [0,1]^d$ and that $\F$ is the class of smooth functions defined by
  $$\F = \mathscr{B}  ([0,1]^{d}) =    \bigcup_{p \in (0,+\infty)} \left( \bigcup_{\substack{\boldsymbol{\alpha} \in (0,+\infty)^{d} \\ \bar{\boldsymbol{\alpha}} >   d (1/p - 1/2)_+}}   \mathscr{B}^{\boldsymbol{\alpha}} (\L^p ([0,1]^{d}))\right).$$ 
It is then shown in~\cite{SartMarkov} that one can choose a collection $\V$ provided by Theorem~1 of~\cite{Akakpo2012} to get:
\begin{eqnarray} \label{eqEpsilonFForBesov}
\text{for all $f \in \mathscr{B}  ([0,1]^{d})$,} \quad \varepsilon_{\F} (f)  \leq  C  \left[\left|f\right|_{\boldsymbol{\alpha},p}^{2 d / (d +  2 \bar{\boldsymbol{\alpha}})} \left(\frac{\log n}{n} \right)^{2 \bar{\boldsymbol{\alpha}}/ (2 \bar{\boldsymbol{\alpha}} + d)}  + \frac{\log n}{n} \right],
\end{eqnarray}
where $p \in (0, +\infty)$, $\boldsymbol{\alpha} \in (0,+\infty)^{d}$, $\bar{\boldsymbol{\alpha}} >   d (1/p - 1/2)_+$ are such that $f \in \mathscr{B}^{\boldsymbol{\alpha}} (\L^p ([0,1]^{d}))$ and where $C > 0$ depends only on $d$, $p$, $\boldsymbol{\alpha}$. 

With this choice of models,  the estimator $\hat{s}$ of Theorem~\ref{thmSelectionModelGeneral} converges  at the expected rate (up to a logarithmic term) for the Hellinger deterministic loss $h$ over a very wide range of possibly inhomogeneous and anisotropic Besov spaces.  It is moreover adaptive with respect to the (possibly unknown) regularity index~$\boldsymbol{\alpha}$ of $\restriction{\sqrt{s}}{[0,1]^d}$.

 \subsubsection*{Regression model.}  
 We can also tackle  the celebrated regression model  $Y_i = g(X_i) + \varepsilon_i$ where~$g$ is an unknown function and where~$\varepsilon_i$ is an unobserved random variable. For the sake of simplicity,  $\XX = \YY = \R$, $A_1 = A_2 = [0,1]$. The conditional density~$s$ is of the form $s(x,y) =  \varphi \left(y - g(x) \right)$ where $\varphi$ is the density of $\varepsilon_i$ with respect to the Lebesgue measure.
   
 Since $\varphi$ and $g$ are unknown, we can, for instance, suppose that these functions are smooth,  which amounts to saying that $\restriction{\sqrt{s}}{[0,1]^2}$ belongs to  
$$\F = \bigcup_{\alpha > 0}  \left\{f ,\, \exists \phi \in \mathcal{H}^{\alpha} (\R), \exists g \in \mathscr{B} ([0,1]),  \|g\|_{\infty} < \infty ,  \, \forall x,y \in [0,1] ,\; f(x,y) = \phi (y - g (x))   \right\}.$$
Here, $ \mathcal{H}^{\alpha} (\R)$ stands for the space of Hölderian functions on $\R$ with regularity index $\alpha \in (0,+\infty)$ and semi norm  $|\cdot|_{\alpha, \infty}$. The notation $\|\cdot\|_{\infty}$ stands for the supremum norm: $ \|g\|_{\infty} = \sup_{x \in [0,1]} |g(x)|$.
An upper bound for $\varepsilon_{\F} (f)$ may be found in Section 4.4 of~\cite{SartMarkov}. Actually, we show in Section~\ref{SectionStructuralAssumptions} that this bound can be slightly improved. To be more precise, the result is the following:  for all   $\alpha > 0$, $p \in (0,+\infty]$, $\beta >  (1/p - 1/2)_+$, $\phi \in \mathcal{H}^{\alpha} (\R)$, $g \in \mathscr{B}^{\beta} (\L^{p} ([0,1]))$, such that $\|g\|_{\infty} < \infty$, and all function $f \in \F$ of the form $f(x,y) = \phi (y - g (x)) $,
\begin{eqnarray} \label{epsilonFPourRegression}
\varepsilon_{\F} (f) \leq   C_1 \left(\frac{\log n}{n}\right)^{\frac{2 \beta (\alpha \wedge 1)}{2 \beta (\alpha \wedge 1) + 1}}  + C_2 \left(\frac{\log n}{n} \right)^{\frac{2 \alpha}{2 \alpha + 1}},
\end{eqnarray}
 where $C_1$ depends only on $p$, $\beta$, $\alpha$, $|g|_{\beta,p}$, $\|g\|_{\infty}$, $|\phi|_{\alpha \wedge 1, \infty}$ and where $C_2$ depends only on $\alpha$, $\|g\|_{\infty}$, $|\phi|_{\alpha, \infty}$.

In particular, if $\phi$ is more regular than $g$ in the sense that $\alpha \geq  \beta \vee 1$, then the rate for estimating the conditional density~$s$ is the same as the one for estimating the regression function~$g$ (up to a logarithmic term). As shown in~\cite{SartMarkov}, this rate is always faster than the rate we would obtain under  smoothness assumptions  only that would ignore the specific form of $s$.

\paragraph{\textbf{Remark.}} The reader could find in~\cite{SartMarkov} a bound for $\varepsilon_{\F}$ when $\F$ corresponds to the heteroscedastic regression model $Y_i = g_1(X_i) + g_2 (X_i) \varepsilon_i$, where $g_1, g_2$ are  smooth unknown functions.

\subsubsection*{A single index type model.}  
In this last example, we investigate the situation in which the explanatory random variables $X_i$  lie in a high dimensional linear space, say $\XX = \R^{d_1}$ with $d_1$ large. On the contrary, the random variables $Y_i$   lie in a small dimensional linear space, say $\YY = \R^{d_2}$ with $d_2$ small. Our aim is then to estimate~$s$ on $A = A_1 \times A_2 = [0,1]^{d_1} \times [0,1]^{d_2}$.

It is well known (and this appears in (\ref{eqEpsilonFForBesov})) that the curse of dimensionality prevents  us to get fast rate of convergence under pure smoothness assumptions on $s$. A solution to overcome this difficulty is to use a single index approach as proposed by~\cite{Hall2005, FanPenYao09}, that is to suppose that the conditional distribution $\mathcal{L} (Y_i \mid X_i = x)$ depends on $x$ through an unknown parameter $\theta \in \R^{d_1}$.    
More precisely, we suppose in this section that $s$ is of the form  ${s(x,y)} = \varphi \left( <\theta, x >, y \right)$ where $<\cdot, \cdot >$ denotes the usual scalar product on $\R^{d_1}$  and where $\varphi$ is a smooth unknown function. 
Without loss of generality, we can suppose that $\theta$ belongs to the unit $\ell^{1}$ ball  of $\R^{d_1}$ denoted by $\mathcal{B}_{1} (0,1)$. We can reformulate these different assumptions by saying that $\restriction{\sqrt{s}}{[0,1]^{d_1+d_2}}$ belongs to the set
\begin{eqnarray*}
\F &=& \bigcup_{\boldsymbol{\alpha} \in (0,+\infty)^{1 + d_2} }  \left\{f ,\, \exists g \in \mathcal{H}^{\boldsymbol{\alpha}} ([0,1]^{1 + d_2}), \exists \theta \in \mathcal{B}_{1} (0,1),  \right. \\
& &  \qquad \qquad\qquad\qquad\qquad \left. \, \forall (x,y) \in [0,1]^{d_1+d_2},\; f(x,y) = g (<\theta, x> , y)   \right\}.
\end{eqnarray*} 
A collection of models $V$ possessing nice approximation properties with respect to the elements~$f$ of $\F$ can be built by using the results of~\cite{BaraudComposite}.  We prove in Section~\ref{SectionStructuralAssumptions} that   we can bound $\varepsilon_{\F} (f)$ as follows: for all  $\boldsymbol{\alpha} \in (0,+\infty)^{1 + d_2} $, $g \in \mathcal{H}^{\boldsymbol{\alpha}} ([0,1]^{1 + d_2})$, $\theta \in \mathcal{B}_{1} (0,1)$,  and all function $f \in \F$ of the form $f(x,y) = g (<\theta, x> , y)$,
\begin{eqnarray} \label{eqEpsilonFPourSingleIndex}
\varepsilon_{\F} (f) \leq   C_1 \left|g \right|_{\boldsymbol{\alpha}, \infty}^{\frac{2(1 +  d_2)}{1 +  d_2 + 2 \bar{\boldsymbol{\alpha}} } } \left(\frac{\log n}{n} \right)^{ \frac{2   \bar{\boldsymbol{\alpha}}}{2   \bar{\boldsymbol{\alpha}} + 1 + d_2} } + C_2  d_1 \frac{\log n \vee \log \left(|g|_{\alpha_1 \wedge 1, \infty}^2 / d_1\right)}{n},
\end{eqnarray} 
where $C_1$ depends only on $d_2$, $\boldsymbol{\alpha}$, and where $C_2$ depends only on $d_2$, $\alpha_1$.   Although $s$ is a function of $d_1 + d_2$ variables,  the rate of convergence of $\hat{s}$  corresponds to the estimation rate of  a smooth function $g$ of   $1 + d_2$ variables only (up to a logarithmic term).  

\section{Proofs} \label{SectionPreuves}

\subsection{Proof of Theorem~\ref{theoremConditionalDensity}.} \label{SectionPreuveThmConditionalDensity}
\begin{lemme} \label{lemmePreuveTheoremConditionalDensity}
For all $f, f' \in S$, and $\xi > 0$,  there exists an event $\Omega_{\xi} (f,f')$ such that  $\P\left[\Omega_{\xi} (f, f')\right] \geq 1 - e^{-n \xi}$ and on which:
\begin{eqnarray*} 
  \left(1 - \varepsilon \right) h^2 \left(s, f'  \right) + T\left(f,f'  \right) \leq\left(1 + \varepsilon \right) h^2 \left(s, f \right)  +  L_1 \xi,
\end{eqnarray*}
where $L_1 > 0$, $\varepsilon \in (0,1)$ are positive universal constants.
\end{lemme}
\begin{proof}
Let $\psi_1$ and $\psi_2$ be the functions defined on $\left(\R_+\right)^2$ by 
\begin{eqnarray*}
\psi_1 (x,y) &=& \frac{\sqrt{y}-\sqrt{x}}{\sqrt{x + y}} \\
\psi_2 (x,y) &=& \sqrt{\frac{x +y }{2}} - \left( \sqrt{x}  + \sqrt{y} \right)
\end{eqnarray*}
  where the convention $0/0 = 0$ is  used.  Let 
  \begin{eqnarray*}
  T_{1,i} (f,f') &=& \psi_1 \left( f(X_i, Y_i), f'(X_i,Y_i)  \right)    \\
  T_{2,i} (f,f') &=& \frac{1}{\sqrt{2}} \int_{A_2} \psi_2 \left(f(X_i,y), f'(X_i,y) \right)  \left( \sqrt{f'(X_i,y)} - \sqrt{f(X_i,y)} \right)  d \mu (y).
  \end{eqnarray*}
  We decompose $T(f,f')$ as  
 \begin{eqnarray*}
T(f,f') = \frac{1}{n} \sum_{i=1}^n  \left(T_{1,i} (f,f')  +  T_{2,i} (f,f') \right)
\end{eqnarray*} 
and  define  $Z(f,f') = T(f,f') - \E[T(f,f')]$.

We need the claim below  whose proof requires the same arguments than those developed in the proofs of    Corollary~1 and Proposition~3 of~\cite{BaraudMesure}. As these arguments are short, we make them explicit at the end of this section to make the paper self contained.
\begin{Claim} \label{ClaimPreuveIneg}
For all $f,f' \in S$,
\begin{eqnarray}
  \left(\sqrt{2}-1 \right) h^2 \left(s, f'  \right) + T\left(f,f'  \right)  &\leq& \left(1+\sqrt{2} \right) h^2 \left(s, f \right) + Z \left(f, f'  \right) \label{eqMajorationTffp}\\
  \E \left[ T_{1,i}^2 \left( f, f'  \right)\right] &\leq& 6 \left[h^2(s,f) + h^2(s,f') \right]. \label{eqMajorationTffpCarre}
\end{eqnarray}
\end{Claim}
  By using Cauchy-Schwarz inequality, 
\begin{eqnarray}
 \left(T_{2,i} (f,f') \right)^2 &\leq&  \left( \int_{A_2} \left(\psi_2 \left(f(X_i,y), f'(X_i,y) \right) \right)^2  \d \mu (y) \right) \times  \nonumber  \\
& & \qquad \qquad\qquad\qquad   \left(  \frac{1}{2}  \int_{A_2}   \left( \sqrt{f'(X_i,y)} - \sqrt{f(X_i,y)} \right)^2 d \mu (y) \right). \label{eqMajorationT2i}
\end{eqnarray}
Note that the function $$z \in [0,\infty) \mapsto \frac{\sqrt{1+z}}{1+\sqrt{z}}$$
is bounded below by $1/\sqrt{2}$ and bounded above by $1$. Therefore, for all $z \geq 0$,
\begin{eqnarray*}
\frac{1}{\sqrt{2}} \leq \frac{\sqrt{1+z}}{1+\sqrt{z}}  \leq 1 &\Longleftrightarrow& \frac{1 + \sqrt{z}}{2} \leq \sqrt{\frac{1+z}{2}} \leq \frac{1+\sqrt{z}}{\sqrt{2}} \\
&\Longleftrightarrow&  -\frac{1 + \sqrt{z}}{2} \leq \sqrt{\frac{1+z}{2}} - (1 + \sqrt{z}) \leq  \frac{1-\sqrt{2}}{\sqrt{2}} \left(1+\sqrt{z}\right) \\
  &\Longrightarrow& \left|\sqrt{\frac{1+z}{2}} - (1 + \sqrt{z}) \right| \leq \frac{1 + \sqrt{z}}{2}.
\end{eqnarray*}
For all $x,y \geq 0$, we derive from this inequality with $z = x/y$ that
\begin{eqnarray*}
\left|\psi_2(x,y) \right|  \leq   \frac{\sqrt{x}+\sqrt{y}}{2} \quad \text{for all $x,y \geq 0$.}
\end{eqnarray*}
Thereby,
\begin{eqnarray*}
\left(\psi_2 (x,y)\right)^2 \leq   \frac{(\sqrt{x}+\sqrt{y})^2}{4} \leq  \frac{x + y}{2},
\end{eqnarray*} 
which together with  $f,f' \in \mathcal{L} (A, \mu)$ and (\ref{eqMajorationT2i}) yields
\begin{eqnarray*}
\E \left[ \left(T_{2,i} (f,f') \right)^2 \right] &\leq&    h^2(f,f') \\
&\leq& 2 h^2(s,f) + 2 h^2 (s,f').
\end{eqnarray*}
By using (\ref{eqMajorationTffpCarre}), we get
\begin{eqnarray*}
\E \left[ \left(T_{1,i} (f,f')  + T_{2,i} (f,f') \right)^2  \right] &\leq& 2 \E \left[ \left(T_{1,i} (f,f')\right)^2  + \left(T_{2,i} (f,f') \right)^2  \right]  \\
&\leq&  16 h^2(s,f) + 16 h^2(s,f').
\end{eqnarray*} 
Now, $T_{1,i} (f,f') \leq 1$ as $\psi_1$ is bounded by $1$ and
\begin{eqnarray*}
T_{2,i} (f,f')    
&\leq& \frac{1}{\sqrt{2}} \int_{A_2}  \left| \frac{\sqrt{f'(X_i,y)} + \sqrt{f'(X_i,y)}}{2}   \right|  \left| \sqrt{f'(X_i,y)} - \sqrt{f(X_i,y)} \right|  d \mu (y)  \\
&\leq&  \frac{1}{\sqrt{2}} \int_{A_2}      \frac{f'(X_i,y)  + f (X_i,y)}{2}  d \mu (y) \\
&\leq& \frac{1}{\sqrt{2}}.
\end{eqnarray*}
Bernstein's inequality and more precisely equation $(2.20)$ of~\cite{Massart2003} shows that for all $\xi > 0$,
$$\P \left[Z(f,f') \leq     \sqrt{32    \left( h^2(s,f) +  h^2(s,f')\right) \xi  } +   (1+1/\sqrt{2}) \xi/3 \right] \geq 1 - e^{- n \xi}.$$
Using now that
$$2 \sqrt{x y} \leq \alpha x + \alpha^{-1} y \quad \text{for all $x,y \geq 0$ and $\alpha > 0$,}$$
we get with probability larger than $1-e^{-n \xi}$,
$$Z(f,f') \leq  \alpha   \sqrt{8}    \left( h^2(s,f) +  h^2(s,f')\right)  +  \left((1+1/\sqrt{2})/3  + \sqrt{8} \alpha^{-1} \right)  \xi.$$
Therefore, we deduce from (\ref{eqMajorationTffp}),
\begin{eqnarray*} 
\left(\sqrt{2}-1 -  \alpha \sqrt{8} \right) h^2 \left(s, f'  \right) + T\left(f,f'  \right)  \leq\left(\sqrt{2}+1  + \alpha \sqrt{8} \right) h^2 \left(s, f \right) +  \left((1+1/\sqrt{2})/3  + \sqrt{8} \alpha^{-1} \right) \xi.
\end{eqnarray*} 
It remains to choose  $\alpha$ to complete the proof. Any value $\alpha \in (0, (\sqrt{2}-1)/ \sqrt{8} )$ works. 
\end{proof}
\begin{lemme} \label{lemmePreuveTheoremConditionalDensity2}
For all   $\xi > 0$ and $f \in S$, there exists an event $\Omega_{\xi} (f) $ such that $\P\left[\Omega_{\xi} (f) \right] \geq 1 - e^{-n \xi}$ and on which:
\begin{eqnarray} \label{eqlemmePreuveTheoremConditionalDensity2} 
\forall f' \in S, \quad \left(1 - \varepsilon \right) h^2 \left(s, f'  \right) + T\left(f,f'  \right) \leq\left(1 + \varepsilon \right) h^2 \left(s, f \right) + L_1 \frac{\bar{\Delta} (f')}{n}  + L_1 \xi,
\end{eqnarray}
where $L_1 > 0$, $\varepsilon \in (0,1)$ are given in Lemma~\ref{lemmePreuveTheoremConditionalDensity}. Moreover, there exists an event $\Omega_{\xi}$ such that  $\P\left[\Omega_{\xi}  \right] \geq 1 - e^{-n \xi}$ and on which:
\begin{eqnarray} \label{eqlemmePreuveTheoremConditionalDensity23}
\quad\quad\quad\quad \forall f, f' \in S,  \; \left(1 - \varepsilon \right) h^2 \left(s, f'  \right) + T\left(f,f'  \right) \leq\left(1 + \varepsilon \right) h^2 \left(s, f \right) + L_1 \frac{\bar{\Delta} (f')}{n} + L_1 \frac{\bar{\Delta} (f)}{n}  + L_1 \xi.
\end{eqnarray}
\end{lemme}
\begin{proof}
The result follows easily from Lemma~\ref{lemmePreuveTheoremConditionalDensity} by setting $$\Omega_{\xi} (f) = \bigcap_{f' \in S} \Omega_{\xi +  \frac{\bar{\Delta}(f') }{n}} (f,f') \quad \text{and} \quad \Omega_{\xi}  = \bigcap_{f \in S} \Omega_{\xi +  \frac{\bar{\Delta}(f)}{n}} (f).  $$
\end{proof}

\begin{lemme} \label{lemmePreuveTheoremConditionalDensity3}
Set $L_0 = (1 + \log 2) L_1$ where $L_1$ is given in Lemma~\ref{lemmePreuveTheoremConditionalDensity}. For all $\xi > 0$,  the following holds with probability larger than $1 - e^{- n \xi}$: if $L \geq  L_0$, for all $f \in S$,
\begin{eqnarray} \label{eqlemmePreuveTheoremConditionalDensity3}
(1-\varepsilon) h^2(s,f)  - R_1 (\xi)  \leq \gamma (f) + L \frac{\bar{\Delta}(f)}{n} \leq (1+\varepsilon) h^2(s,f) + 2 L \frac{\bar{\Delta}(f)}{n} + R_2 (\xi)
\end{eqnarray}
where 
\begin{eqnarray*}
R_1 (\xi)  &=&  \inf_{f' \in S} \left\{ (1+\varepsilon) h^2(s,f') + L \frac{\bar{\Delta}(f')}{n} \right\} + c_1 \xi \\
R_2 (\xi) &=&  - (1- \varepsilon) h^2(s,S) + c_2 \xi
\end{eqnarray*}
and where $\varepsilon$ is given in Lemma~\ref{lemmePreuveTheoremConditionalDensity}, and  $c_1$ and $c_2$ are universal positive constants ($c_1 = 2 L_1$ and $c_2 =  L_1$).
\end{lemme}
\begin{proof}
Let   $g \in S$ be such that
$$(1 + \varepsilon)  h^2 \left(s, g \right)  +  L \frac{\bar{\Delta} (g) }{n}  \leq  \inf_{f' \in S}  \left\{(1 + \varepsilon)  h^2 \left(s, f' \right)  +  L \frac{\bar{\Delta} (f') }{n} \right\} +  L_1 \xi.$$
We shall show that (\ref{eqlemmePreuveTheoremConditionalDensity3}) holds  on the event  $\Omega_{ \xi + \frac{\log 2}{n}} (g)  \cap \Omega_{\xi + \frac{\log 2}{n}}.$
We derive from  (\ref{eqlemmePreuveTheoremConditionalDensity23}) that  for all $f, f' \in S$,  
\begin{eqnarray*}
 \left(1 - \varepsilon \right) h^2 \left(s, f'  \right) + \left(T\left(f,f'  \right) - L \frac{ \bar{\Delta} (f')}{n}\right) &\leq& \left(1 + \varepsilon \right) h^2 \left(s, f \right) +   L_1 \frac{ \bar{\Delta} (f)}{n} + L_1 \frac{ \log 2}{n} + L_1 \xi  \\
 &\leq& \left(1 + \varepsilon \right) h^2 \left(s, f \right) +  L \frac{ \bar{\Delta} (f)}{n} + L_1 \xi,
\end{eqnarray*}
which in particular implies
\begin{eqnarray*}
\gamma(f) &\leq& \left(1 + \varepsilon \right) h^2 \left(s, f \right) +   L \frac{ \bar{\Delta} (f)}{n} - (1- \varepsilon) h^2(s,S) + L_1 \xi \\
&\leq& \left(1 + \varepsilon \right) h^2 \left(s, f \right)  +   L \frac{ \bar{\Delta} (f)}{n} + R_2 (\xi).
\end{eqnarray*} 
This proves the right inequality of  (\ref{eqlemmePreuveTheoremConditionalDensity3}). We now turn to the left one.  We use (\ref{eqlemmePreuveTheoremConditionalDensity2}) to get for all $f \in S$, 
\begin{eqnarray*}
\left(1 - \varepsilon \right) h^2 \left(s, f  \right)  &\leq& \left(1 + \varepsilon \right) h^2 \left(s, g \right) + T\left(f , g  \right) +  L_1 \frac{\bar{\Delta} (f)}{n}  + L_1 \frac{\log 2}{n} + L_1 \xi   \\
&\leq& \left(1 + \varepsilon \right) h^2 \left(s, g \right)  +  L \frac{\bar{\Delta} (g)}{n}  + \left( T\left(f , g  \right) -  L \frac{\bar{\Delta} (g)}{n}  \right) +  L \frac{\bar{\Delta} (f)}{n}  + L_1 \xi  \\
&\leq& \left(1 + \varepsilon \right) h^2 \left(s, g \right)  +  L \frac{\bar{\Delta} (g)}{n}  +  \gamma(f) +  L \frac{\bar{\Delta} (f)}{n}  + L_1 \xi \\
&\leq& \inf_{f' \in S}  \left\{(1 + \varepsilon)  h^2 \left(s, f' \right)  +  L \frac{\bar{\Delta} (f') }{n} \right\}  + \gamma(f)  +  L \frac{\bar{\Delta} (f)}{n} + 2 L_1 \xi.
\end{eqnarray*}
This ends the proof.
\end{proof}
The computations preceding Theorem~\ref{theoremConditionalDensity} finally complete its proof. \qed

\begin{proof}  [Proof of Claim~\ref{ClaimPreuveIneg}] 
Define the function $g = (f+f')/2$ and the measure $\zeta$ by $$\d \zeta(x,y) = f_X(x) \d \nu (x) \d \mu (y).$$ We have,
\begin{eqnarray*}
\E \left[T\left(f,f'  \right) \right] &=& \frac{1}{\sqrt{2}} \int_A \frac{\sqrt{f'} - \sqrt{f}}{\sqrt{g}} s  \d \zeta  + \frac{1}{\sqrt{2}} \int_{A} \sqrt{g} \left(\sqrt{f'} - \sqrt{f} \right)  \d \zeta + \frac{1}{\sqrt{2}} \int_{A}  \left({f} - {f'} \right)   \d \zeta \\
&=& \frac{1}{\sqrt{2}}   \left(\int_A \sqrt{\frac{{f'}}{{g}}} s  \d \zeta  + \int_{A} \sqrt{g f'} \d \zeta -  \int_{A}  {f'}   \d \zeta \right) \\
& & \quad -  \frac{1}{\sqrt{2}}  \left(  \int_A \sqrt{\frac{{f}}{{g}}} s  \d \zeta  +\int_{A} \sqrt{g f} \d \zeta -   \int_{A}  {f}   \d \zeta \right).
\end{eqnarray*}
Now,
\begin{eqnarray} \label{eqPreuveEsperanceT2maj}
h^2(s,f') - h^2(s,f) &=&  \left(\int_{A} \sqrt{s f} \d \zeta - \frac{1}{2}\int_{A} f   \d \zeta \right) -\left(\int_{A} \sqrt{s  f' } \d \zeta  - \frac{1}{2}\int_{A} f'  \d \zeta \right) \nonumber \\
&=& -\frac{1}{\sqrt{2}} \E \left[T(f,f')\right] + \frac{1}{2} \left(\int_A \sqrt{\frac{{f'}}{{g}}} s  \d \zeta  + \int_{A} \sqrt{g f'} \d \zeta -  2 \int_{A} \sqrt{s  f' } \d \zeta  \right) \nonumber  \\
& & - \frac{1}{2} \left(\int_A \sqrt{\frac{{f}}{{g}}} s  \d \zeta  + \int_{A} \sqrt{g f} \d \zeta -  2 \int_{A} \sqrt{s  f } \d \zeta  \right) \nonumber \\ 
&=& -\frac{1}{\sqrt{2}} \E \left[T(f,f')\right]  + \frac{1}{2} \int_{A} \sqrt{\frac{{f'}}{{g}}} \left(\sqrt{s} - \sqrt{g} \right)^2 \d \zeta    - \frac{1}{2} \int_{A} \sqrt{\frac{{f}}{{g}}} \left(\sqrt{s} - \sqrt{g} \right)^2 \d \zeta \nonumber \\
&\leq&  -\frac{1}{\sqrt{2}} \E \left[T(f,f')\right]   + \frac{1}{2} \int_{A} \sqrt{\frac{{f'}}{{g}}} \left(\sqrt{s} - \sqrt{g} \right)^2 \d \zeta. \nonumber
\end{eqnarray}
By using  $\sqrt{f'/g} \leq \sqrt{2}$,  and a concavity argument,
 \begin{eqnarray*} \label{eqPreuveEsperanceT21}
\frac{1}{2} \int_{A} \sqrt{\frac{{f'}}{{g}}} \left(\sqrt{s} - \sqrt{g} \right)^2 \d \zeta  &\leq&  \sqrt{2} h^2(s,g) \\
&\leq& \frac{1}{\sqrt{2}} \left(h^2(s,f) + h^2(s,f') \right).
 \end{eqnarray*}
We now derive from
    $-\E \left[T(f,f')\right] = -T(f,f') + Z(f,f')$ that,
 \begin{eqnarray*}
h^2(s,f') - h^2(s,f) &\leq& \frac{-T(f,f') + Z(f,f')}{\sqrt{2}} + \frac{1}{\sqrt{2}} \left(h^2(s,f) + h^2(s,f') \right).
\end{eqnarray*}
This proves (\ref{eqMajorationTffp}).

We now turn to the proof of (\ref{eqMajorationTffpCarre}): 
\begin{eqnarray*}
\E \left[ T_{1,i}^2 \left( f, f'  \right)\right] &=& \frac{1}{2} \int_{A} \frac{\left(\sqrt{f'} - \sqrt{f}\right)^2}{g} s \d \zeta \\
 &=& \frac{1}{2} \int_{A}  \left(\sqrt{f'} - \sqrt{f}\right)^2 \left(\frac{\sqrt{s}-\sqrt{g}}{\sqrt{g}} + 1 \right)^2 \d \zeta \\
  &\leq&   \int_{A}  \frac{\left(\sqrt{f'} - \sqrt{f}\right)^2}{g} \left({\sqrt{s}-\sqrt{g}}\right)^2 \d \zeta + \int_{A}  \left(\sqrt{f'} - \sqrt{f}\right)^2 \d \zeta.
\end{eqnarray*}
Now, $(\sqrt{f'} - \sqrt{f})^2/{g} \leq 2$ and hence,
\begin{eqnarray*}
\E \left[ T_{1,i}^2 \left( f, f'  \right)\right]  
  &\leq&  2 \int_{A}   \left({\sqrt{s}-\sqrt{g}}\right)^2 \d \zeta + \int_{A}  \left(\sqrt{f'} - \sqrt{f}\right)^2 \d \zeta  \\
  &\leq& 4 h^2(s,g) + 2 h^2(f,f') \\
  &\leq& 4 h^2(s,g) + 4 h^2(s,f) + 4 h^2(s,f').
\end{eqnarray*}
By using a concavity argument, $h^2(s,g) \leq 1/2 (h^2(s,f) + h^2(s,f'))$. Finally,
\begin{eqnarray*}
\E \left[ T_{1,i}^2 \left( f, f'  \right)\right]     \leq  6 \left[ h^2(s,f) + h^2(s,f') \right],
\end{eqnarray*}
which proves (\ref{eqMajorationTffpCarre}).
\end{proof}

\subsection{Proof of Proposition~\ref{propRisqueDeStilde}.}
As $f$ belongs to $\mathbb{L}_+^1 (A, \nu \otimes \mu)$, Fubini's theorem says that there exists $A_1' \subset A_1$ such that  $\nu (A_1 \setminus A_1') = 0$ and such that
$$\forall x \in A_1', \quad  \int_{A_2} f(x,y) \d \mu (y) < \infty.$$
Let  $(\mathbb{L}^2 (A_2,  \mu), \|\cdot\|)$  be the linear space of square integrable functions on $A_2$ with respect to $\mu$.  

For all $x \in A_1'$,  $\sqrt{f(x, \cdot)}$ belongs to $\mathbb{L}^2 (A_2,  \mu)$ and 
\begin{eqnarray*} 
\sqrt{ \pi (f) (x,y)} = \frac{\sqrt{ f (x,y)}}{\max \left(\| \sqrt{f (x, \cdot)}\|, 1 \right) } \quad \text{for all $(x,y) \in A_1' \times A_2$}.
\end{eqnarray*}
Note that $\sqrt{ \pi (f) (x,\cdot)}$  is the projection of $\sqrt{f (x,\cdot)}$ onto the unit ball $\{g \in \mathbb{L}^2 (A_2,  \mu), \, \left\|g\right\| \leq 1 \}$. As the projection is Lipschitz continuous,
\begin{eqnarray*}
\left\|\sqrt{\pi({s}) (x,\cdot)} - \sqrt{\pi (f) (x,\cdot)}  \right\|^2  \leq   \left\|\sqrt{{s} (x,\cdot)} - \sqrt{f(x,\cdot)}  \right\|^2 \quad \text{for all $x \in A_1'$.}
\end{eqnarray*}
As   $\|\sqrt{s (x, \cdot)} \| \leq 1$,  $\sqrt{\pi({s}) (x,\cdot)} = \sqrt{s (x, \cdot)}$ and hence
\begin{eqnarray*}
\left\|\sqrt{s (x, \cdot)} - \sqrt{\pi (f) (x,\cdot)}  \right\|^2  \leq   \left\|\sqrt{s (x, \cdot)} - \sqrt{f(x,\cdot)}  \right\|^2 \quad \text{for all $x \in A_1'$.}
\end{eqnarray*}
By integrating both inequalities with respect to $x$, 
\begin{eqnarray*}
 \int_{A_1'} \int_{A_2} \left( \sqrt{s (x, \cdot)} - \sqrt{\pi (f) (x,y)}  \right)^2 \d \nu (x) \d \mu (y) &\leq& \int_{A_1'} \int_{A_2} \left( \sqrt{s (x, \cdot)} - \sqrt{f (x,y)}  \right)^2 \d \nu (x) \d \mu (y) \\
&\leq&  2  h^2(s, f).
\end{eqnarray*} 
Since $\nu (A_1 \setminus A_1') = 0$, the left-hand side of the above inequality is merely $2 h^2(s, \pi (f))$, which proves the proposition. \qed

\subsection{Proof of Proposition~\ref{propRisqHelIntegreeHisto1}.}
Let for each $K \in m$, $I_K \subset A_1$ and $J_K \subset A_2$ be such that $K = I_K \times J_K$.
Let $\mathcal{I} = \left\{I_K, \, K \in m \right\}$, and for each $I \in \mathcal{I}$, let $\mathcal{J}_I = \left\{J, \, I \times J \in m\right\}$.
 The partition $m$ can be rewritten as
 $$m = \bigcup_{I \in \mathcal{I} } \left\{I \times J ,\,  J \in \mathcal{J}_I \right\}.$$
 Remark that $$|m| = \sum_{I \in \mathcal{I}} |\mathcal{J}_I|.$$
 We now introduce for all $I \in \mathcal{I}$ and $J \in \mathcal{J}_I$, 
 \begin{eqnarray*}
 N(I \times J) = \sum_{i=1}^n \1_{I } (X_i) \1_{J} (Y_i) \quad \text{and} \quad   M(I) = \sum_{i=1}^n \1_I (X_i).
 \end{eqnarray*}
With these notations, the estimator $\hat{s}_m$ becomes
  $$\hat{s}_m = \sum_{\substack{I \in \mathcal{I} \\ J \in \mathcal{J}_I} } \frac{N(I \times J)}{M(I) \mu(J)} \1_{I \times J}.$$
We define 
\begin{eqnarray*} 
\bar{s}_m  &=& \sum_{\substack{I \in \mathcal{I} \\ J \in \mathcal{J}_I} } \frac{\E[N(I \times J)]}{\E[M(I)] \mu(J)} \1_{I \times J}, \\
{s}^{\star}_m  &=& \sum_{\substack{I \in \mathcal{I} \\ J \in \mathcal{J}_I} } \frac{N(I \times J)}{\E[M(I)] \mu(J)} \1_{I \times J}.
\end{eqnarray*} 
We use the triangular inequality to get
\begin{eqnarray} \label{eqPreuvePropHistoTerm0}
\E \left[h^2(s , \hat{s}_m ) \right] \leq 2 h^2(s, \bar{s}_m) +  4 \E \left[h^2(\bar{s}_m, {s}^{\star}_m ) \right] + 4 \E \left[h^2({s}^{\star}_m,\hat{s}_m )\right].
\end{eqnarray} 
It remains to control both of the three terms appearing in the right-hand side of the above inequality.  The first term can be upper bounded thanks to Lemma~2 of~\cite{BaraudBirgeHistogramme}:
\begin{eqnarray} \label{eqPreuvePropHistoTerm1}
h^2(s, \bar{s}_m) \leq 2 h^2(s, V_m).
\end{eqnarray}
Now,
 \begin{eqnarray*}
h^2(\bar{s}_m, s^{\star}_m) &=& \frac{1}{2 n } \sum_{\substack{I \in \mathcal{I} \\ J \in \mathcal{J}_I}} \left( \sqrt{\frac{\E[N(I \times J)]}{\E[M(I)] \mu(J)}} - \sqrt{  \frac{N(I \times J)}{\E[M(I)] \mu(J)}}\right)^2 \E [M(I)] \mu(J) \\
&=& \frac{1}{2 n} \sum_{\substack{I \in \mathcal{I} \\ J \in \mathcal{J}_I}}\left( \sqrt{\E[N(I \times J)]}- \sqrt{ N(I \times J)}\right)^2 \\
&\leq& \frac{1}{2 n } \sum_{\substack{I \in \mathcal{I} \\ J \in \mathcal{J}_I}}\frac{\left(  \E[N(I \times J)] -   N(I \times J) \right)^2}{  \E[N(I \times J)] }.
\end{eqnarray*}
By taking the expectation of both sides,
\begin{eqnarray} \label{eqPreuvePropHistoTerm2}
\E \left[h^2(\bar{s}_m, s^{\star}_m)\right] \leq \frac{1}{2 n} \sum_{\substack{I \in \mathcal{I} \\ J \in \mathcal{J}_I}}\frac{ \var \left[N(I \times J) \right]}{  \E[N(I \times J)] } \leq  \frac{1}{2 n} \sum_{\substack{I \in \mathcal{I} \\ J \in \mathcal{J}_I}} 1 \leq \frac{|m|}{2 n}.
\end{eqnarray}
As to the third term,
\begin{eqnarray*}
h^2(\hat{s}_m, s^{\star}_m)  &=& \frac{1}{2 n}   \sum_{\substack{I \in \mathcal{I} \\ J \in \mathcal{J}_I} } \left( \sqrt{\frac{N(I \times J)}{M(I) \mu(J)}} - \sqrt{  \frac{N(I\times J)}{\E[M(I)] \mu(J)}}\right)^2 \E [M(I)] \mu (J)     \\
&=& \frac{1}{2 n}   \sum_{\substack{I \in \mathcal{I} \\ J \in \mathcal{J}_I} } \left( \sqrt{{\E[M(I)]}} - \sqrt{M(I)} \right)^2  \frac{N(I \times J)}{M(I)} \\
&=& \frac{1}{2 n }   \sum_{I \in \mathcal{I}  } \left( \sqrt{{\E[M(I)]}} - \sqrt{M(I)} \right)^2  \sum_{J \in \mathcal{J}_I}  \frac{N(I \times J)}{M(I)}\\
&\leq& \frac{1}{2 n}   \sum_{I \in \mathcal{I}} \left|\mathcal{J}_I\right| \left( \sqrt{{\E[M(I)]}} - \sqrt{M(I)} \right)^2.
\end{eqnarray*}
Therefore,
\begin{eqnarray}
\E \left[h^2(\hat{s}_m, s^{\star}_m)  \right] &\leq& \frac{1}{2 n} \sum_{I \in \mathcal{I}} \left|\mathcal{J}_I\right|  \E \left[ \frac{\left(M(I) - \E[M(I)]\right)^2}{\E[M(I)]} \right]  \nonumber \\
&\leq& \frac{1}{2 n} \sum_{I \in \mathcal{I}}   \left|\mathcal{J}_I\right| \frac{\text{var} [M(I)]}{\E[M(I)]} \nonumber \\
&\leq&  \frac{1}{2 n} \sum_{I \in \mathcal{I}}  \left|\mathcal{J}_I\right|  \nonumber \\
&\leq& \frac{|m|}{2 n}. \label{eqPreuvePropHistoTerm3}
\end{eqnarray}
Gathering (\ref{eqPreuvePropHistoTerm0}), (\ref{eqPreuvePropHistoTerm1}), (\ref{eqPreuvePropHistoTerm2}) and (\ref{eqPreuvePropHistoTerm3}) leads to the result. \qed

\subsection{Proof of Proposition~\ref{propRisqHelIntegreeHisto2}.}
We use in this proof the notations introduced in the proof of Proposition~\ref{propRisqHelIntegreeHisto1}. We derive from the triangular inequality  and (\ref{eqPreuvePropHistoTerm1}),  
\begin{eqnarray} \label{eqDansPreuvepropRisqHelIntegreeHisto2}
h^2(s , \hat{s}_m )  \leq 4 h^2(s,  V_m) +  4 h^2(\bar{s}_m, {s}^{\star}_m )  + 4 h^2({s}^{\star}_m,\hat{s}_m ),
\end{eqnarray} 
and it remains to bound  the two last terms from above. Yet,
\begin{eqnarray*}
h^2(\bar{s}_m, s^{\star}_m)  &=& \frac{1}{2 n} \sum_{\substack{I \in \mathcal{I} \\ J \in  \mathcal{J}_I}} \left( \sqrt{\E[N(I \times J)]}- \sqrt{ N(I \times J)}\right)^2.
\end{eqnarray*}
Theorem~8 of~\cite{BaraudBirgeHistogramme} (applied with $A = 1$ and $\kappa = 1$) shows that, for all $x > 0$, with probability larger than $1 - e^{-x}$,
$$\sum_{\substack{I \in \mathcal{I} \\ J \in  \mathcal{J}_I}} \left( \sqrt{\E[N(I \times J)]}- \sqrt{ N(I \times J)}\right)^2\leq 8 |m| + 202   x.$$
Now, 
\begin{eqnarray*}
h^2(\hat{s}_m, s^{\star}_m)  &=& \frac{1}{2 n}   \sum_{\substack{I \in \mathcal{I} \\ J \in \mathcal{J}_I} } \left( \sqrt{\frac{N(I \times J)}{M(I) \mu(J)}} - \sqrt{  \frac{N(I\times J)}{\E[M(I) \mu(J)]}}\right)^2 \E [M(I)] \mu (J)     \\
&=& \frac{1}{2 n }   \sum_{\substack{I \in \mathcal{I} \\ J \in \mathcal{J}_I} } \left( \sqrt{{\E[M(I)]}} - \sqrt{M(I)} \right)^2  \frac{N(I \times J)}{M(I)} \\
&=& \frac{1}{2 n }   \sum_{I \in \mathcal{I}  } \left( \sqrt{{\E[M(I)]}} - \sqrt{M(I)} \right)^2  \sum_{J \in \mathcal{J}_I}  \frac{N(I \times J)}{M(I)}\\
&\leq& \frac{1}{2 n}   \sum_{I \in \mathcal{I}} \left( \sqrt{{\E[M(I)]}} - \sqrt{M(I)} \right)^2.
\end{eqnarray*}
A new application of Theorem~8 of~\cite{BaraudBirgeHistogramme} shows that  for all $x > 0$, with probability larger than $1 - e^{-x}$,
$$ \sum_{I \in \mathcal{I}} \left( \sqrt{{\E[M(I)]}} - \sqrt{M(I)} \right)^2 \leq 8 |\mathcal{I}| + 202   x.$$
We then deduce from (\ref{eqDansPreuvepropRisqHelIntegreeHisto2}) that for all $x > 0$, with probability larger than $1 - 2 e^{-x}$,
\begin{eqnarray*}
h^2(s, \hat{s}_m)  \leq 4 h^2(s,  V_m)  + 16  \frac{ |m| + |\mathcal{I}|}{n}+ 808 \frac{x}{n}
\end{eqnarray*}
The result follows with $x = n \xi + \log 2$. \qed

\subsection{Proof of Theorem~\ref{thmSelectionModelGeneral}.}
Let, for each model $V \in \V$,     $T_V$ be a subset   of $V$  satisfying the two following conditions:
\begin{enumerate}[-]
\item for all $g \in V$, there exists $f \in T_V$ such that $d_2 (f,g) \leq 1/n$
\item  $$\left|\left\{f \in T_V, \, d_2(f,0) \leq 2 \right\} \right| \leq \left(4  n  + 1 \right)^{\dim V}.$$
\end{enumerate}
For instance, we can define $T_V$ as a maximal $1/n$-separated subset of $V$ in the metric space $(\L^2(A, \nu \otimes \mu), d_2)$ in the sense of Definition~5 of~\cite{BirgeTEstimateurs}.  The bound on the cardinality is then given by  Lemma~4 of~\cite{BirgeTEstimateurs}.
Let
$$S_V = \left\{f_+^2, \; f \in T_V, \, d_2(f,0) \leq 2 \right\} \cup \{0\} \quad \text{and} \quad S = \bigcup_{V \in \V} S_V.$$  
We now define the map $\bar{\Delta}$  on $S$ by
$$\bar{\Delta} (f) = \inf_{\substack{ V \in \V \\ S_V  \ni f }}  \left\{\Delta (V) + \log |S_V | \right\}.$$
Without loss of generality, we can assume that $S \subset \mathcal{L} (A,  \mu)$ (thanks to Proposition~\ref{propRisqueDeStilde}). Theorem~\ref{theoremConditionalDensity} shows that there exists an estimator $\hat{s}$  satisfying for all $\xi > 0$ and probability larger than $1 - e^{- n \xi}$,
\begin{eqnarray*}
 h^2(s,\hat{s} )  \leq  c_1 \inf_{f \in S} \left\{ h^2 (s,f) + L \frac{\bar{\Delta}(f)}{n}\right\}  + c_2\xi.
\end{eqnarray*}
Hence,
\begin{eqnarray*}
 h^2(s,\hat{s} )  &\leq&  c_1 \inf_{V \in \V  } \left\{ h^2 (s,S_V) + L \frac{\Delta(V) + \log |S_V|}{n}\right\}  + c_2\xi \\
 &\leq& c_1 \inf_{V \in \V  } \left\{ \kappa \inf_{ \substack{f \in T_V \\ d_2(f,0) \leq 2}}  d_2^2 (\sqrt{s}, f) + L \frac{\Delta(V) + \log |S_V|}{n}\right\}  + c_2\xi.
\end{eqnarray*}
As $d_2 (\sqrt{s}, 0) \leq 1$ and $0 \in S_V$,  
\begin{eqnarray*}
\inf_{\substack{f \in T_V \\ d_2(f,0) \leq 2}}  d_2 (\sqrt{s}, f) &=&  \inf_{ \substack{f \in T_V}}  d_2 (\sqrt{s}, f) \\
&\leq& \inf_{ \substack{f \in V}}  d_2 (\sqrt{s}, f) + 1/n.
\end{eqnarray*} 
Therefore,
\begin{eqnarray*}
 h^2(s,\hat{s} )  &\leq& c_1 \inf_{V \in \V  } \left\{ 2\kappa   d_2^2 (\sqrt{s}, V) + 2 \frac{\kappa}{n^2}  + L \frac{\Delta(V) + \log \left(1 + (4  n + 1)^{\dim V} \right)}{n}\right\}  + c_2\xi  \\
 &\leq& C \left(  \inf_{V \in \V  } \left\{ \kappa   d_2^2 (\sqrt{s}, V)   +  \frac{\Delta(V) + (\dim V) \log n}{n}\right\}   +  \frac{\kappa}{n^2} + \xi  \right)
\end{eqnarray*}
for $C$ large enough. \qed

\subsection{Structural assumptions.} \label{SectionStructuralAssumptions}  
Theorem~2 and Corollary~1  of~\cite{BaraudComposite} are useful tools to deal with structural assumptions. They show  how to build collections $\V$ of linear spaces $V$ with good approximation properties with respect to composite functions $f$ of the form $f = g \circ u$.
Using these results is the strategy of~\cite{SartMarkov} to get bounds on $\varepsilon_{\F} (f)$  for classes $\F$ corresponding to structural assumptions on $s$.  Nevertheless, this direct application of the results of~\cite{BaraudComposite} (with $\tau = \log n / n$)  leads to an unnecessary additional logarithmic term in the risk bounds.  A careful look at the proof of Theorem~2 of~\cite{BaraudComposite}  shows that the following result holds.

\begin{thm}
Suppose that Assumption~\ref{hypDensiteDesXiCasGeneral} holds and that $\nu \otimes \mu (A) = 1$.  Let $l \in \N^{\star}$ and $\mathcal{L}^{\infty} ([0,1]^l)$ be  the set of bounded functions on $[0,1]^l$ endowed with the supremum distance
$$d_{\infty} (g_1,g_2) = \sup_{x \in [0,1]^l} |g_2(x) - g_1(x)| \quad \text{for $g_1, g_2 \in \mathcal{L}^{\infty} ([0,1]^l)$.}$$
Let $\mathcal{U}$ be the set of  functions $u = (u_1,\dots,u_l)$ going from $A$ to $[0,1]^l$ and
$$\F = \left\{g \circ u, \; g   \in \bigcup_{\boldsymbol{\alpha} \in (0,1]^{l} }\mathcal{H}^{\boldsymbol{\alpha}} ([0,1]^l), \, u \in \mathcal{U} \right\}.$$
 Let $\mathbb{F}$ be an at most countable collection of finite dimensional linear subspaces $F$ of $\mathcal{L}^{\infty} ([0,1]^l)$ endowed with a map $\Delta_{\mathbb{F}} \geq 1$ satisfying $$\sum_{F \in \mathbb{F}} e^{- \Delta_{\mathbb{F}} (F)} \leq 1.$$
Let, for all $j \in \{1,\dots, l\}$, $\mathbb{T}_j$ be  an at most countable collection of  subsets $T$ of $\L^2 (A, \nu \otimes \mu)$. We assume that each $T$ is either a unit set, or a finite dimensional linear space. If $T$ is a singleton, we set $\dim T = 0$. If $T$ is a non trivial linear space, $\dim T$ stands for its usual linear dimension. We endow $\mathbb{T}_j$   with a non-negative map $\Delta_{\mathbb{T}_j} $ satisfying $$\sum_{T \in \mathbb{T}_j} e^{- \Delta_{\mathbb{T}_j} (T)} \leq 1.$$ 
Then, there exist a collection $\V$ and a map $\Delta$ such that for all function $f \in \F$ of the form $f = g \circ u$, with   $g \in \mathcal{H}^{\boldsymbol{\alpha}} ([0,1]^l)$ for some $\boldsymbol{\alpha} = (\alpha_1,\dots, \alpha_l) \in (0,1]^{l}$, and    $u = (u_1,\dots, u_l) \in \mathcal{U}$,
\begin{eqnarray*}
C \varepsilon_{\F} (f)  &\leq& \sum_{j=1}^l \inf_{T \in \mathbb{T}_j}  \left\{l  |g_j|_{\alpha_j}^2 d_2^{2 \alpha_j} (u_j, T) + \frac{\Delta_{\mathbb{T}_j} (T) + (\dim T) \mathcal{L}_{j,T} }{n} \right\}  \\
& & \quad + \inf_{F \in \mathbb{F}} \left\{d_{\infty}^2  (g, F) + \frac{\Delta_{\mathbb{F}} (F) + (\dim F) \log n}{n} \right\}.
\end{eqnarray*} 
 In the above inequality, $C$ is a positive universal constant, $g_j$, $|g_j|_{\alpha_j}$ are defined as explained in Section~\ref{SectionDefinitionHolderSpace},  and $  \mathcal{L}_{j,T}$ is defined when $\dim T > 0$ by
 \begin{eqnarray*}
  \mathcal{L}_{j,T}  &=&  \left[\alpha_j^{-1} \log \left(n l   |g_j|_{\alpha_j}^2 / \dim T \right) \right] \vee 1 \\
  &\leq&  C' \left[ \log n \vee \log \left(  |g_j|_{\alpha_j}^2 / \dim T \right) \vee 1 \right]
 \end{eqnarray*} 
 for $C'$  depending only on $l$ and $\alpha_j$. When $\dim T = 0$,  $  \mathcal{L}_{j,T} = 1$.
\end{thm}  
The proof of (\ref{epsilonFPourRegression}) is almost the same as the one of Corollary~4 of~\cite{SartMarkov}.  The only difference is that we apply the above theorem in place of Theorem~2 of~\cite{BaraudComposite} with $\tau = \log n / n$.

We now turn to the proof of (\ref{eqEpsilonFPourSingleIndex}).  Note that a function $f$ of the form $f(x,y) =  g (<\theta, x> , y)$ can be rewritten as $f(x,y) = g ( u_1 (x,y), u_2(x,y), \dots, u_{1+d_2} (x,y))$ where $u_1 (x,y) = <\theta, x>$ and $u_{j} (x,y) = y$ for $j \in \{2,\dots,1+d_2\}$. There exists   a pair $(\mathbb{F}, \Delta_{\mathbb{F}})$ such that for all  $\boldsymbol{\alpha} \in (0,+\infty)^{1 + d_2} $, $g \in \mathcal{H}^{\boldsymbol{\alpha}} ([0,1]^{1 + d_2})$, 
$$ \inf_{F \in \mathbb{F}} \left\{d_{\infty}^2  (g, F) + \frac{\Delta_{\mathbb{F}} (F) + (\dim F) \log n}{n} \right\} \leq C_1 \left[\left|g \right|_{\boldsymbol{\alpha}, \infty}^{\frac{2(1 +  d_2)}{1 +  d_2 +  2 \bar{\boldsymbol{\alpha}} } } \left(\frac{\log n}{n} \right)^{ \frac{2   \bar{\boldsymbol{\alpha}}}{2   \bar{\boldsymbol{\alpha}} + 1 + d_2} } +  \frac{\log n}{n}  \right]$$
for a constant $C_1$ depending only on $d_2$, $\boldsymbol{\alpha}$ (see, for instance, \cite{BaraudComposite}).  Let for $\theta \in  \R^{d_1}$, $u_{\theta}$ be the function defined  by $u_{\theta}(x,y) = <\theta, x>$ and $T_1$ be the linear space  defined by  $T_1 = \{u_{\theta},   \theta \in \R^{d_1}\}$.     We use the above theorem with $l = 1+d_2$, $\TT _1= \{T\}$, $\Delta_{\TT_1} (T) = 1$, $\TT_j = \{ \{u_j\}\}$, $\Delta_{\TT_j} (\{u_j\}) = 0$ for $j \in \{2,\dots, 1+d_2\}$ to derive  that for all  $\boldsymbol{\alpha} \in (0,+\infty)^{1 + d_2} $, $g \in \mathcal{H}^{\boldsymbol{\alpha}} ([0,1]^{1 + d_2})$, $\theta \in \mathcal{B}_{1} (0,1)$,  and all function $f \in \F$ of the form $f(x,y) = g (<\theta, x> , y)$,
\begin{eqnarray*}
\varepsilon_{\F} (f)  \leq  C_1 \left[\left|g \right|_{\boldsymbol{\alpha}}^{\frac{2(1 +  d_2)}{1 +  d_2 +  2\bar{\boldsymbol{\alpha}} } } \left(\frac{\log n}{n} \right)^{ \frac{2   \bar{\boldsymbol{\alpha}}}{2   \bar{\boldsymbol{\alpha}} + 1 + d_2} } +  \frac{\log n}{n}  \right] + C_2 d_1    \frac{\log n  \vee \log \left(|g|_{\alpha_1 \wedge 1}^2 / d_1\right)}{n} 
\end{eqnarray*} 
where $C_1$ depends only on $d_2$, $\boldsymbol{\alpha}$ and $C_2$ depends only on $\alpha_1 \wedge 1$, $d_2$. \qed

 \bibliography{biblio}
 \bibliographystyle{apalike}
\end{document}